\documentclass[11pt,reqno,final]{amsart}

\usepackage[utf8]{inputenc}
\usepackage[marginpar=2cm]{geometry}
\usepackage{setspace}
\usepackage{indentfirst}
\usepackage{tgtermes}
\RequirePackage{lineno}
\usepackage{mathrsfs,mathtools,amssymb,stmaryrd,tikz-cd}
\usepackage{enumitem}

\theoremstyle{plain}
\newtheorem{introthm}{Theorem}

\newtheorem{theorem}{Theorem}[section]
\newtheorem{lemma}[theorem]{Lemma}
\newtheorem{proposition}[theorem]{Proposition}
\newtheorem{corollary}[theorem]{Corollary}

\theoremstyle{definition}
\newtheorem{definition}[theorem]{Definition}
\newtheorem{example}[theorem]{Example}

\theoremstyle{remark}
\newtheorem{remark}[theorem]{Remark}

\usepackage[backend=bibtex,bibencoding=utf8,style=numeric,
url=false,doi=false,eprint=true,isbn=false,
hyperref=auto,backref=false]{biblatex}
\usepackage[notcite,notref,color]{showkeys}

\usepackage[pagebackref=false,hidelinks,unicode=true,bookmarks=true,
linktoc=page,pdfstartview={FitH},final=true]{hyperref}

\allowdisplaybreaks[3]
\binoppenalty=\maxdimen
\relpenalty=\maxdimen
\DeclareMathOperator{\id}{id}
\DeclareMathOperator{\pr}{pr}
\DeclareMathOperator{\ev}{ev}

\DeclareMathOperator{\Hom}{Hom}

\DeclareMathOperator{\Der}{Der}
\DeclareMathOperator{\coDer}{coDer}
\DeclareMathOperator{\CE}{CE}
\DeclareMathOperator{\tw}{tw}
\DeclareMathOperator{\sh}{sh}
\DeclareMathOperator{\lin}{lin}
\DeclareMathOperator{\con}{con}
\DeclareMathOperator{\HH}{HH}

\DeclareMathOperator{\Lie}{Lie}

\newcommand{\cA}{\mathcal{A}}
\newcommand{\cC}{\mathcal{C}}

\newcommand{\cL}{\mathcal{L}}

\newcommand{\XX}{\mathfrak{X}}

\newcommand{\ZZ}{\mathbb{Z}}

\newcommand{\CC}{\mathbb{C}}
\newcommand{\KK}{\mathbb{K}}

\newcommand{\tensor}{\otimes}
\newcommand{\tangent}[1]{T{#1}}
\newcommand{\tangentp}[2]{T_{#1}{#2}}
\newcommand{\liederivative}[1]{\cL_{#1}}
\newcommand{\tliederivative}[1]{\widetilde{\cL}_{#1}}
\newcommand{\into}{\hookrightarrow}
\newcommand{\onto}{\twoheadrightarrow}
\newcommand{\xto}{\xrightarrow}

\newcommand{\conv}{\star}
\newcommand{\degree}[1]{|{#1}|}
\newcommand{\argument}{-}
\newcommand{\sign}{\varepsilon}
\newcommand{\shuffle}{\sh}

\newcommand{\bfv}{\mathbf{v}}
\newcommand{\bfw}{\mathbf{w}}
\newcommand{\transpose}{T}

\newcommand{\fG}{\mathfrak{G}}
\newcommand{\fUV}{\mathfrak{U}^{V}}
\newcommand{\act}{\mathbin{\mkern 1mu \vcenter{\hbox{\scalebox{0.6}{$\blacklozenge$}}} \mkern 1mu}}

\title{Linearisation, splitting property and homotopy algebras}

\author{Seokbong Seol}
\address{School of Mathematics, Korea Institute for Advanced Study}
\email{azuredream89@kias.re.kr}

\author{Kai Wang}
\address{School of Mathematical Sciences, University of Science and Technology of China}
\email{wangkai17@ustc.edu.cn}

\thanks{The first author is supported by the KIAS Individual Grant MG090802 at Korea Institute for Advanced Study. The second author is supported by the National Key R\&D Program of China (No. 2024YFA1013803) and the NSFC (No. 12501053).}

\begin{document}

\begin{abstract}
We study the linearisation of formal vector fields and homotopy algebras using graded coalgebras. 
We prove that a formal vector field is linearisable if and only if it satisfies a splitting property, and we give an explicit recursive construction of a linearising coordinate transformation. This provides a direct coalgebraic proof of Bandiera's characterisation of linearisable, equivalently homotopy abelian, $L_{\infty}[1]$ algebras. We establish an associative analogue, characterising linearisability of $A_{\infty}[1]$ algebras by a splitting property for one-sided Hochschild cochains. We also examine the corresponding splitting property for standard Hochschild cochains and show that it is equivalent to $A_{\infty}[1]$ commutativity in the sense of Briggs and Gélinas. Finally, we establish a corresponding criterion for formal endomorphisms and give sufficient conditions for splitting that recover Basto-Gonçalves' linearisation theorem for admissible formal vector fields.
\end{abstract}

\maketitle

\tableofcontents


\section*{Introduction}

Given a finite dimensional graded vector space $V$, an $L_{\infty}[1]$ algebra structure on $V$ can be viewed as a derivation $Q$ on the graded algebra of formal power series $\widehat{S}(V^{\vee})$ on $V$. 
If we view $V$ as a graded analogue of a formal manifold (in the sense that its algebra of functions is $\widehat{S}(V^{\vee})$), then $Q$ is viewed as a formal vector field on $V$. By the structure of $L_{\infty}[1]$ algebras, $Q$ has degree $1$, $Q^{2}=0$, and $Q$ vanishes at the origin $0\in V$. 

In~\cite{MR3622306}, Bandiera introduced a notion of ``splitting property'' for an $L_{\infty}[1]$ algebra and showed that an $L_{\infty}[1]$ algebra $(V,Q)$ satisfies the splitting property if and only if there exists an isomorphism of $L_{\infty}[1]$ algebras $\Phi:(V,Q)\to (V,q)$ on the same space $V$, where $q$ is the linear component of $Q$. His method relies on the properties of $L_{\infty}[1]$ algebras, especially those from higher derived brackets~\cite{MR2163405, MR2223157} and the $L_{\infty}[1]$ algebra structure on the mapping cone of a morphism of DG Lie algebras~\cite{MR2361936}.

Motivated by this result, we show that the theorem of Bandiera extends to any formal manifold $V$ equipped with an arbitrary formal vector field $Q$ vanishing at $0\in V$, without assuming that $Q$ has degree $1$ or satisfies $Q^{2}=0$; we only use properties of coalgebras and their coderivations.
We introduce the notion of a splitting property in this setting and show that $(V,Q)$ satisfies the splitting property if and only if $Q$ is linearisable (i.e., there exists a formal change of coordinates of $V$ which transforms $Q$ into a linear vector field).

Note that the space of formal vector fields $\XX(V)=\Der(\widehat{S}(V^{\vee}))$ is naturally equipped with the Lie derivative $\liederivative{Q}$ along $Q$, and the space of constant vector fields $\XX_{\con}(V)$ is naturally equipped with the Lie derivative $\liederivative{q}$ along the linear component $q$ of $Q$. 
The space of constant vector fields $\XX_{\con}(V)$ can be naturally identified with the fibre $T_{0}V\cong V$ of the tangent bundle $TV$ over the origin $0\in V$, 
and thus the evaluation map $\sigma_{0}:\XX(V)\to \XX_{\con}(V)$ at $0\in V$ intertwines the Lie derivatives: $\liederivative{q} \sigma_{0}=\sigma_{0} \liederivative{Q}$. 
We say the pair $(V,Q)$ satisfies the \textit{splitting property} if there exists a degree-preserving map $\tau:\XX_{\con}(V)\to \XX(V)$ such that $\sigma_{0}\tau= \id_{\XX_{\con}(V)}$ and $\tau \liederivative{q} = \liederivative{Q} \tau$.

The terminology \textit{splitting property} arises from the fact that, when $Q$ arises from an $L_{\infty}[1]$ algebra, the pairs $(\XX(V),\liederivative{Q})$ and $(\XX_{\con}(V),\liederivative{q})$ are both cochain complexes, and the existence of such $\tau$ is equivalent to saying that the following short exact sequence of cochain complexes splits:
\[\begin{tikzcd}
0 \arrow{r} & \ker \sigma_{0} \arrow{r}  & (\XX(V), \liederivative{Q}) \arrow{r}{\sigma_{0}}& (\XX_{\con}(V), \liederivative{q})\arrow{r} \arrow[dotted, bend left]{l}{\tau} & 0  \, .
\end{tikzcd}
\]

We use the language of coalgebras and coderivations---the dual notions of algebras and derivations---to prove the following:
\begin{introthm}[Theorem~\ref{thm:First}]\label{thm:A}
Let $V$ be a formal manifold, and let $Q$ be a formal vector field on $V$ satisfying $\sigma_{0}(Q)=0$. Then $Q$ is linearisable if and only if $Q$ satisfies the splitting property.
\end{introthm}
An analogous theorem holds for formal endomorphisms of $V$ (see~\ref{thm:Second}).

Theorem~\ref{thm:A} provides another viewpoint on the formal linearisation problem.
Geometrically, the $\widehat{S}(V^{\vee})$-linear extension of the splitting map $\tau$ induces a change of frame for the formal tangent bundle. The splitting property expresses the compatibility of this change of frame with the vector fields $Q$ and $q$.
A frame arising from formal coordinates necessarily consists of commuting vector fields. Our criterion does not impose this requirement on the frame determined by $\tau$. Using coalgebras and coderivations, we provide an explicit recursive construction of a formal coordinate transformation $\psi:V\to V$ directly from the splitting map $\tau$. 
The induced Lie algebra isomorphism $\psi_{\ast}$ satisfies $\psi_{\ast}(q)=Q$, although its restriction to $\XX_{\con}(V)$ need not coincide with $\tau$. Thus, the existence of a compatible frame is sufficient for linearisability, even when that frame does not itself arise from coordinates.

The relation between the splitting property and linearisability continues to hold for homotopy algebras. 
For $L_{\infty}[1]$ algebras, the graded coalgebra version of Theorem~\ref{thm:A} (see Theorem~\ref{thm:First-1}) provides an alternative proof of Bandiera's theorem: an $L_{\infty}[1]$ algebra is linearisable if and only if it satisfies the splitting property. We prove an analogous statement for $A_{\infty}[1]$ algebras. To formulate it, we introduce the corresponding splitting property using one-sided Hochschild cochains.

Let $(A,\mu)$ be an $A_{\infty}[1]$ algebra. The operations $\mu$ induce the natural right $A_{\infty}[1]$ $A$-module structure on $A$. Denote by $A^{R}$ the $A_{\infty}[1]$ $A$-$A$-bimodule on $A$ obtained by extending the right module structure trivially on the left: all bimodule operations involving an algebra input to the left of the module input vanish. 
We consider a cochain complex $(\Hom(TA,A),D_{\mu})$, where $TA$ denotes the tensor coalgebra of $A$, and $D_{\mu}$ is the Hochschild differential with coefficients in $A^{R}$; its cohomology is $\HH(A,A^{R})$.

We say an $A_{\infty}[1]$ algebra $(A,\mu)$ satisfies the \textit{splitting property} if there exists a degree-preserving map $\tau:A\to \Hom(TA,A)$ such that 
\[\tau(a)(1)=a \, , \qquad  D_{\mu}\tau(a)=\tau(\mu_{1}(a)) \]
for every $a\in A$, where $\mu_{1}$ denotes the linear component of $\mu$.
Recall that an $A_{\infty}[1]$ algebra $(A,\mu)$ is called \textit{linearisable} if $(A,\mu)$ is isomorphic to its underlying cochain complex $(A,\mu_{1})$ as an $A_{\infty}[1]$ algebra.
\begin{introthm}[Theorem~\ref{thm:ATFAE}]\label{thm:B}
An $A_{\infty}[1]$ algebra is linearisable if and only if it satisfies the splitting property.
\end{introthm}
We also prove that the conditions in Theorem~\ref{thm:B} are equivalent to the $E_{1}$-degeneration of the spectral sequence of $(\Hom(TA,A),D_{\mu})$ with respect to the standard filtration $F^{p}(\Hom(TA,A))=\Hom(T^{\geq p}A,A)$.

Theorem~\ref{thm:B} uses the Hochschild cochain complex with coefficients in the bimodule $A^{R}$. 
It is natural to ask whether the same criterion holds if we replace $(\Hom(TA,A),D_{\mu})$ with $(\coDer(TA),\liederivative{\mu})$, the standard Hochschild cochain complex computing $\HH(A,A)$. The answer is negative in general. In this case, the analogous splitting condition is equivalent to the $A_{\infty}[1]$ commutativity of a minimal model in the sense of Briggs and Gélinas~\cite{arXiv:1702.00721} (see Proposition~\ref{prop:CTA}).
However, the $E_{1}$-degeneration of the spectral sequence of $(\coDer(TA),\liederivative{\mu})$ with respect to the standard filtration is equivalent to the linearisability of $(A,\mu)$, rather than $A_{\infty}[1]$ commutativity (see remark~\ref{rem:Last}).

As an application, we also give sufficient conditions for the existence of a splitting map for formal vector fields, recovering Basto-Gonçalves' linearisation theorem for admissible formal vector fields~\cite{MR2678982}. In the associative setting, we briefly illustrate the criterion with Börjeson's $A_{\infty}[1]$ algebras~\cite{MR3624036}, whose splitting maps have a particularly simple form, recovering their linearisability established by Markl~\cite{MR3385702}.

\medskip
\noindent
\textbf{Notation and conventions.}

Throughout this paper, let $\KK$ be a field of characteristic $0$. Unless otherwise specified, all vector spaces, linear maps, and tensor products are taken over $\KK$. All gradings are assumed to be $\ZZ$-gradings.
For graded vector spaces $V$ and $W$, we write $V \otimes W$ for their graded tensor product and $\Hom(V,W)$ for the space of graded linear maps. 
If $\dim V < \infty$ or $\dim W < \infty$, we often identify 
\[
\Hom(V,W) \cong V^\vee \otimes W.
\]

To handle signs arising from graded structures, we use the Koszul sign rule: interchanging two homogeneous elements $a$ and $b$ introduces a factor $(-1)^{\degree{a}\degree{b}}$, where $\degree{a}$ and $\degree{b}$ denote their degrees.

For a graded vector space $V$, we denote by $T^n V = V^{\otimes n}$ its $n$-fold tensor power. The symmetric group $\mathbb{S}_n$ acts on $T^n V$ by
\[\sigma(v_1 \otimes \cdots \otimes v_n)
= \varepsilon(\sigma; v_1,\dots,v_n)\, v_{\sigma^{-1}(1)} \otimes \cdots \otimes v_{\sigma^{-1}(n)},\]
where $\varepsilon(\sigma; v_1,\dots,v_n)$ is the Koszul sign determined by permuting the homogeneous elements $v_1,\dots,v_n$. We define the symmetric power $S^n V := V^{\odot n}$ as the quotient space of $T^n V$ by this $\mathbb{S}_{n}$-action.

Finally, for a DG vector space, we adopt the cohomological degree convention: if $(V,d_{V})$ is a DG vector space, then the differential $d_{V}$ is an operator of degree $+1$.

\section{Linearisations in formal geometry}\label{Section:Linearisation of formal vector fields}
\subsection{Formal vector fields}

A formal manifold (or a formal pointed manifold) is a finite dimensional vector space $V$ whose algebra of functions is 
the algebra of formal power series on $V$:
\[ \widehat{S}(V^{\vee})=\prod_{k=0}^{\infty} S^{k}(V^{\vee}) \, . \]
A morphism $\psi:V\to W$ between formal manifolds $V$ and $W$
is defined as a morphism of associative algebras 
$\Psi:\widehat{S}(W^{\vee}) \to \widehat{S}(V^{\vee})$ continuous with respect to the $I$-adic topology---that is, for $I_{W}= 
\prod_{k=1}^{\infty} S^{k}(W^{\vee})$ and similarly for $V$, it satisfies
\[\Psi (I_{W})\subset  I_{V}\, .\]
We note that, geometrically, each morphism of formal manifolds $\psi:V\to W$ preserves the origin, i.e., $\psi(0)=0$. We also note that the associated algebra morphism $\Psi:\widehat{S}(W^{\vee}) \to \widehat{S}(V^{\vee})$ is uniquely determined by the sequence of linear maps $\psi_{k}:W^{\vee} \to S^{k}(V^{\vee})$ for $k\geq 1$.

We say $\psi:V\to W$ is a formal diffeomorphism if $\Psi$ is an isomorphism of associative algebras. When $V=W$, a formal diffeomorphism $\psi$ can be seen as a formal change of coordinates of $V$. Note that $\Psi$ is an isomorphism if and only if its linear component $\psi_{1}:W^{\vee}\to V^{\vee}$ is an isomorphism of vector spaces.

The space
\[\XX(V)=\Der(\widehat{S}(V^{\vee})) \cong \widehat{S}(V^{\vee})\tensor V\, \]
is naturally equipped with the Lie algebra structure $[-, - ]$ induced by the commutator on $\Der(\widehat{S}(V^{\vee}))$.
Elements of $\XX(V)$ are formal vector fields.
Each formal vector field $X\in \XX(V)$ has a natural decomposition 
\[X=X_{0}+X_{1}+\cdots \, ,\]
where $X_{k}\in S^{k}(V^{\vee})\tensor V$. We refer to $X_0$ as the constant part of $X$, and to $X_{1}$ as its linear part.
The space of all constant formal vector fields is denoted by $\XX_{\con}(V)$ and is naturally identified with $V\cong \tangentp{0}{V}$, where $T_{0}V$ denotes the fibre of the formal tangent bundle $TV$ of $V$ over the origin $0\in V$. 
We denote the projection to the constant formal vector fields by 
\begin{equation}\label{eq:Sig0}
\sigma_{0}:\XX(V)\to \XX_{\con}(V)\cong V, \qquad X\mapsto X_{0} \, .
\end{equation}
 Geometrically, $\sigma_{0}$ can be viewed as the evaluation map at the origin $0\in V$. 
We also denote
the space of all linear formal vector fields by $\XX_{\lin}(V)$. 

For brevity, we will often drop the adjective `formal' when referring to diffeomorphisms and vector fields.

A vector field $X$ with $\sigma_{0}(X)=0$ is called \textbf{linearisable} if there exists a diffeomorphism $\psi:V\to V$ such that the induced map $\psi_{\ast}:\XX(V)\to \XX(V)$ satisfies
\begin{equation}\label{eq:Linearisable}
	\psi_{\ast}(X)\in \XX_{\lin}(V) \subset \XX(V) \, .
\end{equation}

Given any vector field $X$,
it is straightforward to check that the commutator $[X_{1},\argument]:\XX(V)\to \XX(V)$ with the linear part $X_{1}$ stabilises $\XX_{\con}(V)\cong V$. That is, $[X_{1},-] : V\to V$.
Moreover, when $\sigma_{0}(X)=0$, the evaluation map $\sigma_{0}$ intertwines with $[X,-]$ and $[X_{1},-]$, so that the following diagram commutes:
\[
\begin{tikzcd}
	\XX(V) \arrow{r}{\sigma_{0}} \arrow[swap]{d}{[X,-]} & V \arrow{d}{[X_{1},-]}\\
	\XX(V)\arrow{r}{\sigma_{0}} & V
\end{tikzcd}
\]

\begin{definition}\label{defn:SP1}
Let $V$ be a formal manifold, and let $X$ be a formal vector field on $V$ satisfying $\sigma_{0}(X)=0$.
We say $X$ satisfies the \textbf{splitting property} if there exists a linear map $\tau:V\to \XX(V)$ such that
\begin{equation}
\sigma_{0} \circ \tau = \id_{V},\qquad [X,\tau(v)]=\tau([X_{1},v])
\end{equation}
for all $v\in V$. We call such $\tau$ a splitting map for $X$.
\end{definition}
We note that the $\widehat{S}(V^{\vee})$-linear extension of the splitting map $\tau$ can be viewed as an $\widehat{S}(V^{\vee})$-module isomorphism from $\XX(V)$ onto itself. Geometrically, the splitting map $\tau$ can be viewed as a change of frames for the formal tangent bundle $TV$.

The following is our main result.
\begin{theorem}\label{thm:First}
	Let $V$ be a formal manifold, and 
	let $X$ be a formal vector field satisfying $\sigma_{0}(X)=0$. 
	The following are equivalent.
	\begin{enumerate}[label=\rm{(\roman*)}]
		\item \label{item:a}
		The formal vector field $X$ is linearisable. 
		\item \label{item:b}
		There exists a formal diffeomorphism $\psi:V\to V$ satisfying
		\[ \psi(0)=0, \quad  \tangent{\psi}|_{0}=\id_{\tangentp{0}{V}}, \quad \psi_{\ast} (X_{1}) = X \, ,\]
		where $X_{1}$ is the linear component of $X$. 
		\item \label{item:c}
		The formal vector field $X$ satisfies the splitting property.
	\end{enumerate}
\end{theorem}

\begin{remark}
As noted in the introduction, the induced map $\psi_{\ast}:\XX(V)\to \XX(V)$ must preserve the Lie algebra structure on $\XX(V)$, while the $\widehat{S}(V^{\vee})$-linear extension of the splitting map $\tau$ does not necessarily preserve the Lie algebra structure.
\end{remark}

The equivalence between \ref{item:a}~and~\ref{item:b} is standard.

We prove \ref{item:b}~$\Leftrightarrow$~\ref{item:c} in a slightly more general setting using the language of graded coalgebras: the underlying space $V$ is a (not necessarily finite dimensional) graded vector space. However, for Theorem~\ref{thm:First}, we will always assume that $V$ is finite dimensional; this discrepancy arises due to the fact that, unlike the finite dimensional case, not all derivations on $\widehat{S}(V^{\vee})$ arise as the dual of coderivations on $SV$ if $\dim V=\infty$.

Now, let $V$ be a (not necessarily finite dimensional) graded vector space. 
As described in Section~\ref{sec:coAlgSV}, the graded space $SV$ equipped with $\Delta:SV\to SV\tensor SV$ is the symmetric coalgebra of $V$. 
Let $\Phi:SV\to SV$ be an endomorphism of the graded coalgebra $SV$. Then $\Phi$ is uniquely determined by $\phi_{k}:S^{k}(V)\to V$, for $k\geq 1$.
Similarly, let $Q$ be a coderivation on $SV$. 
Then $Q$ is uniquely determined by the sequence $\{q_{n}\}_{n\geq 0}$ of maps $q_{n}:S^{n}(V)\to V$. 
Denote by $\widetilde{q}_{n}$ the coderivation on $SV$ determined by $q_{n}:S^{n}(V)\to V$. 
Then $Q=\widetilde{q}_{0}+\widetilde{q}_{1}+\cdots$.

When $V$ is a finite dimensional ordinary vector space (viewed as a graded vector space concentrated in degree $0$), Theorem~\ref{thm:First-1} below proves \ref{item:b}~$\Leftrightarrow$~\ref{item:c} of Theorem~\ref{thm:First}. 
Indeed, the diffeomorphism $\psi$ in Theorem~\ref{thm:First}~\ref{item:b} is obtained by setting $\Psi=\Phi^{\transpose}$ as the dual of the isomorphism $\Phi$ in Theorem~\ref{thm:First-1}~\ref{item:A} below, and Theorem~\ref{thm:First}~\ref{item:c} is equivalent to Theorem~\ref{thm:First-1}~\ref{item:B} below, under the identification
\[\Der(\widehat{S}(V^{\vee})) \cong\widehat{S}(V^{\vee})\tensor V \cong  \coDer (SV) \, .\]

\begin{theorem}\label{thm:First-1}
Let $V$ be a graded vector space, and let $Q$ be a coderivation on $SV$ satisfying $q_{0}=0$.
With the notation above, the following are equivalent.
\begin{enumerate}[label=\rm{(\roman*)}]
	\item There exists an isomorphism of graded coalgebras
	\[ \Phi:SV\to SV\]
	such that 
	$\phi_{1}=\id_{V}$ 
	and
	$Q\circ \Phi = \Phi \circ \widetilde{q}_{1}$. \label{item:A}
	\item There exists a degree-preserving linear map $\tau:V\to \coDer(SV)$ such that $\tau(v)\cdot 1 = v$ and 
	\[
	\tau(q_{1}(v)) \cdot \bfv =\liederivative{Q}\big(\tau(v)\big) \cdot \bfv
	\]
	for all $v\in V$ and $\bfv \in SV$.
	\label{item:B}
\end{enumerate}
\end{theorem}

As a graded coalgebra counterpart of Definition~\ref{defn:SP1}, we say the graded coderivation $Q$ satisfies the \textbf{splitting property} if it satisfies Theorem~\ref{thm:First-1}~\ref{item:B}. Such $\tau$ is called the splitting map for $Q$.

To prove Theorem~\ref{thm:First-1}, we need some preparation. 
Denote by $\phi_{n}^{k}$ the composition of $\Phi$ with the natural inclusion and projection
\[
\phi_{n}^{k}:S^{n}V\into SV \xto{\Phi} SV \onto S^{k}V \, ,
\]
and in particular, $\phi_n:=\phi_n^1$.
Then Theorem~\ref{thm:First-1}~\ref{item:A} is equivalent to $\phi_{1}=\id_{V}$ and 
\begin{equation}\label{eq:PhiQ}
	\phi_{n} \widetilde{q}_{1}=\sum_{k=1}^{n} q_{k}\phi_{n}^{k}
\end{equation}
for all $n\geq 1$.

Using the identification $\coDer(SV)\cong \Hom(SV,V)$, we may write $\tau=\sum_{n=0}^{\infty} \tau_{n}$ where, for each $n\geq 0$,
\[\tau_{n}\in \Hom(V\tensor S^{n}(V), V).\]
For simplicity of notation, we assume that $Q$ is homogeneous, and
for each $v\in V$ and $\bfv \in S^{n}(V)$, we write
$\tau_{n}(v\tensor \bfv )=\tau_{n}^{v}(\bfv)$. Also, we denote by $\widetilde{\tau}_{n}^{v}$ the coderivation on $SV$ determined by $\tau_{n}^{v}:S^{n}(V)\to V$.
It is straightforward to check that Theorem~\ref{thm:First-1}~\ref{item:B} is equivalent to $\tau^{v}_{0}(1)=v$ and

\begin{equation}\label{eq:Qn}
	\tau^{q_{1}(v)}_{n}(\bfv)
	= q_{n+1}(v\odot \bfv) +\sum_{k=1}^{n} 
	\big(q_{n-k+1}(\widetilde{\tau}^{v}_{k}(\bfv)) -(-1)^{\degree{v}\degree{Q}}\tau^{v}_{k}(\widetilde{q}_{n-k+1}(\bfv))\big)
\end{equation}
for all $n \geq 1$. 
Using Sweedler's notation $\Delta \bfv= \bfv_{(1)}\tensor \bfv_{(2)}$,
Eq.~\eqref{eq:Qn} can be written as
\[\tau^{q_{1}(v)}_{n}(\bfv)
= q_{n+1}(v\odot \bfv) +\sum_{k=1}^{n}
\big(q_{n-k+1}(\tau^{v}_{k}(\bfv_{(1)})\odot \bfv_{(2)}) -(-1)^{\degree{v}\degree{Q}}\tau^{v}_{k}(q_{n-k+1}(\bfv_{(1)})\odot \bfv_{(2)})\big) \, .
\]
Now we are ready to prove Theorem~\ref{thm:First-1}.

\begin{proof}[Proof of Theorem~\ref{thm:First-1}]
	Assume that~\ref{item:A} holds. Define $\tau:V\to \coDer(SV)$ by
	\[ \tau(v) \cdot \bfv = \Phi(v\odot \Phi^{-1}(\bfv)) \]
	for $v\in V$ and $\bfv \in SV$. Since $\Phi(1)=1$ and $\Phi(v)=v$, we have 
	$\tau(v):1\mapsto v$. Moreover, since 
	\[q_{1}(v)\odot \Phi^{-1}(\bfv) = \widetilde{q}_{1}( v\odot \Phi^{-1}(\bfv) ) -(-1)^{\degree{Q}\degree{v}} v\odot \widetilde{q}_{1}(\Phi^{-1}(\bfv)) \]
	for $\bfv \in SV$, we have
	\begin{eqnarray*}
		\liederivative{Q}(\tau(v)) \cdot \bfv 
		&= &Q(\Phi(v\odot \Phi^{-1}(\bfv))) - (-1)^{\degree{Q}\degree{v}} \Phi(v\odot \Phi^{-1}(Q(\bfv))) \\
		&=& \Phi \big(\widetilde{q}_{1}(v\odot \Phi^{-1}(\bfv)) - (-1)^{\degree{Q}\degree{v}} v\odot \widetilde{q}_{1}(\Phi^{-1}(\bfv))\big) \\
		&=& \Phi(q_{1}(v)\odot \Phi^{-1}(\bfv)) \\
		&=& \tau(q_{1}(v)) \cdot \bfv \, ,
	\end{eqnarray*}
	hence obtain~\ref{item:B}.
	
	Conversely, assume that \ref{item:B} holds.
	Let $\phi_{1}(v)=v$ for $v\in V$ and, for $n\geq 1$,
	\begin{equation}\label{eq:PN1}
		\phi_{n+1}(\bfv )=\frac{1}{n+1} \sum_{i}\sum_{k=1}^{n} \sign_{i} \cdot \tau_{k}^{v_{i}}\phi_{n}^{k}(\bfv^{\{i\}}) \, .
	\end{equation}
	where $\bfv=v_{1}\odot \cdots \odot v_{n+1}\in S^{n+1}(V)$ and $\sign_{i}=(-1)^{\degree{v_{i}}(\degree{v_{1}}+\cdots+\degree{v_{i-1}})}$.
	In particular, for $v, w \in V$, we have $\phi_{2}(v\odot w)=\frac{1}{2} (\tau_{1}^{v}(w)+(-1)^{\degree{v}\degree{w}}\tau_{1}^{w}(v))$.
	
	We use induction on $n$. Clearly, since $\phi_{1}=\id_{V}$, we have $\phi_{1}q_{1}=q_{1}\phi_{1}$.
	Assume that 
	\begin{equation}\label{eq:PhiQAs}
		\phi_{r} \widetilde{q}_{1}=\sum_{k=1}^{r}q_{k}\phi_{r}^{k}
	\end{equation}
	holds for all $r\leq n$. By~\eqref{eq:PhiQ}, it suffices to prove that $\phi_{n+1} \widetilde{q}_{1}=\sum_{k=1}^{n+1}q_{k}\phi_{n+1}^{k}$.
	
	By~\eqref{eq:PN1}, we have
	\begin{eqnarray*}
		&&( \phi_{n+1}\widetilde{q}_{1}-q_{1}\phi_{n+1} )(\bfv) \\
		&= &\frac{1}{n+1} \sum_{i}\sum_{k=1}^{n}\sign_{i} \cdot
		\big(\tau_{k}^{q_{1}(v_{i})}\phi_{n}^{k}(\bfv^{\{i\}})
		+(-1)^{\degree{v_{i}}\degree{Q}} \tau_{k}^{v_{i}}\phi_{n}^{k} \widetilde{q}_{1}(\bfv^{\{i\}})
		- q_{1}\tau_{k}^{v_{i}}\phi_{n}^{k}(\bfv^{\{i\}}) \big) \\
		&=& \frac{1}{n+1} \sum_{i}\sum_{k=1}^{n}\sign_{i} \cdot
		\big((\tau_{k}^{q_{1}(v_{i})}
		+(-1)^{\degree{v_{i}}\degree{Q}}\tau_{k}^{v_{i}} \widetilde{q}_{1} - q_{1}\tau_{k}^{v_{i}})\phi_{n}^{k}(\bfv^{\{i\}}) + 
		(-1)^{\degree{v_{i}}\degree{Q}}\tau_{k}^{v_{i}}(\phi_{n}^{k} \widetilde{q}_{1} - \widetilde{q}_{1}\phi_{n}^{k})(\bfv^{\{i\}})\big) .
	\end{eqnarray*}
	By assumption, we use~\eqref{eq:Qn} to obtain
	\[ (\tau_{1}^{q_{1}(v_{i})}+(-1)^{\degree{v_{i}}\degree{Q}}\tau_{1}^{v_{i}}q_{1} - q_{1}\tau_{1}^{v_{i}})\phi_{n}(\bfv^{\{i\}})=q_{2}(v_{i}\odot \phi_{n}(\bfv^{\{i\}}))\]
	and, for $k\geq 2$, 
	\begin{eqnarray*}
		&&(\tau_{k}^{q_{1}(v_{i})}+(-1)^{\degree{v_{i}}\degree{Q}}\tau_{k}^{v_{i}}\widetilde{q}_{1} - q_{1}\tau_{k}^{v_{i}})\phi_{n}^{k}(\bfv^{\{i\}}) 
		\\
		&=&q_{k+1}(v_{i}\odot \phi_{n}^{k}(\bfv^{\{i\}}))
		+\sum_{m=2}^{k} \big( q_{m}\widetilde{\tau}_{k-m+1}^{v_{i}}\phi_{n}^{k}(\bfv^{\{i\}}) 
		-(-1)^{\degree{v_{i}}\degree{Q}}\tau_{k-m+1}^{v_{i}} \widetilde{q}_{m}\phi_{n}^{k}(\bfv^{\{i\}}) \big) \, . 
	\end{eqnarray*}

	Using the notation from Appendix~\ref{sec:coAlgSV}, we have
	\begin{eqnarray*}
		\sum_{k=2}^{n}\sum_{m=2}^{k}\tau_{k-m+1}^{v_{i}}\widetilde{q}_{m}\phi_{n}^{k}
		&=&\sum_{l=1}^{n-1}\sum_{m=2}^{n-l+1}\tau_{l}^{v_{i}}\widetilde{q}_{m}\phi_{n}^{m+l-1} \\
		&=&\sum_{l=1}^{n-1}\tau_{l}^{v_{i}} \Big( \sum_{m=2}^{n-l+1} \sum_{r=m}^{n-l+1} q_{m}\phi_{r}^{m}\conv \phi_{n-r}^{l-1}\Big)\\
		&=&\sum_{l=1}^{n-1}\tau_{l}^{v_{i}} \Big( \sum_{r=2}^{n-l+1}\sum_{m=2}^{r}q_{m}\phi_{r}^{m}\conv \phi_{n-r}^{l-1}\Big)
		\\
		&=&\sum_{l=1}^{n-1}\tau_{l}^{v_{i}} \Big( \sum_{r=1}^{n-l+1}(\phi_{r}\widetilde{q}_{1}-q_{1}\phi_{r})\conv \phi_{n-r}^{l-1}\Big)\\
		&=&\sum_{k=1}^{n} \tau_{k}^{v_{i}}(\phi_{n}^{k}\widetilde{q}_{1}-\widetilde{q}_{1}\phi_{n}^{k}) \, .
	\end{eqnarray*}
	where the second equality holds by Lemma~\ref{lem:QNPNK}, the fourth equality holds by assumption~\eqref{eq:PhiQAs} and $\phi_{1}=\id_{V}$, and the last equality holds by $\phi_{n}^{n}=\id_{S^{n}(V)}$, Lemma~\ref{lem:PNKQ} and~\ref{lem:QNPNK}.

	Combining these, we have shown that
	\[
	( \phi_{n+1}\widetilde{q}_{1}-q_{1}\phi_{n+1} )(\bfv) = 
	\frac{1}{n+1} \sum_{i}\sign_{i} \cdot
	\Big\{ \sum_{k=1}^{n} q_{k+1}(v_{i}\odot \phi_{n}^{k}(\bfv^{\{i\}}))+\sum_{k=2}^{n}\sum_{m=2}^{k}q_{m}\widetilde{\tau}_{k-m+1}^{v_{i}}\phi_{n}^{k}(\bfv^{\{i\}})\Big\}  \, .
	\]
	
	Note that 
	\[
	\sum_{k=2}^{n} \sum_{m=2}^{k}
	q_{m}\widetilde{\tau}_{k+1-m}^{v_{i}}\phi_{n}^{k}(\bfv^{\{i\}}) 
	= \sum_{m=2}^{n} \sum_{l=1}^{n-m+1} q_{m}\widetilde{\tau}^{v_{i}}_{l}\phi^{l+m-1}_{n}(\bfv^{\{i\}}) \, ,
	\]
	and for each fixed $m$, we have
	\begin{equation}\label{eq:rPLUS1}
		\begin{aligned}
		\sum_{i}\sum_{l=1}^{n-m+1}\sign_{i} \cdot  \widetilde{\tau}_{l}^{v_{i}} \phi^{l+m-1}_{n}(\bfv^{\{i\}})
		&=\sum_{i}\sum_{l=1}^{n-m+1} \sum_{r=l}^{n-m+1} 
		\sign_{i}\cdot (\tau_{l}^{v_{i}}\phi_{r}^{l}\conv \phi_{n-r}^{m-1})(\bfv^{\{i\}})\\
		&= \sum_{i}\sum_{r=1}^{n-m+1}\sum_{l=1}^{r} 
		\sign_{i} \cdot (\tau_{l}^{v_{i}}\phi_{r}^{l} \conv \phi_{n-r}^{m-1})(\bfv^{\{i\}})\\
		& = \sum_{r=1}^{n-m+1}(r+1)\cdot (\phi_{r+1}\conv \phi_{n-r}^{m-1})(\bfv)
	\end{aligned}
\end{equation}
	where the first equality holds by applying the coderivation $\tau_{l}^{v_{i}}$ (in place of $q_{l}$) to Lemma~\ref{lem:QNPNK}, and the last equality holds by the relation~\eqref{eq:PN1}.

	Together with the relation
	\[\sum_{i} \sign_{i} \cdot q_{k+1}(v_{i}\odot \phi_{n}^{k}(\bfv^{\{i\}})) = q_{k+1}(\phi_{1}\conv \phi_{n}^{k})(\bfv),\]
	Lemma~\ref{lem:PNM} implies that
	\begin{eqnarray*}
		( \phi_{n+1}\widetilde{q}_{1}-q_{1}\phi_{n+1} )(\bfv) &=&
		\frac{1}{n+1} \Big\{ \sum_{k=1}^{n}q_{k+1}(\phi_{1}\conv \phi_{n}^{k}) + \sum_{m=2}^{n} \sum_{r=1}^{n-m+1}(r+1)\cdot q_{m}(\phi_{r+1}\conv \phi_{n-r}^{m-1})\Big \} (\bfv)\\
		&=& \frac{1}{n+1} \sum_{m=2}^{n+1} \sum_{r=0}^{n-m+1}q_{m}((r+1)\cdot \phi_{r+1}\conv \phi_{n-r}^{m-1})(\bfv) \\
		&=& \sum_{m=2}^{n+1}q_{m} \phi_{n+1}^{m}(\bfv) \, .
	\end{eqnarray*}
	This completes the proof.
\end{proof}

We emphasise that the linearising isomorphism $\Phi$ from Theorem~\ref{thm:First-1} can be constructed explicitly.
\begin{corollary}\label{cor:MainFormula}
With the notation from Theorem~\ref{thm:First-1} and the preceding paragraphs, assume that the splitting map $\tau:V\to \coDer(SV)$ for $Q$ is given. Then the coalgebra isomorphism $\Phi=(\phi_{1},\phi_{2},\ldots):SV\to SV$ that linearises $Q$ can be constructed by an explicit recursive formula
$\phi_{1}=\id_{V}$ and for $n\geq 1$,
\[\phi_{n+1}(\bfv )=\frac{1}{n+1} \sum_{i}\sum_{k=1}^{n} \sign_{i} \cdot \tau_{k}^{v_{i}}\phi_{n}^{k}(\bfv^{\{i\}}) \]
where $\bfv=v_{1}\odot \cdots \odot v_{n+1}\in S^{n+1}(V)$ and $\sign_{i}=(-1)^{\degree{v_{i}}(\degree{v_{1}}+\cdots+\degree{v_{i-1}})}$.
\end{corollary}

\begin{remark}\label{rem:AltForm}
The coalgebra isomorphism $\Phi:SV\to SV$ constructed recursively in Corollary~\ref{cor:MainFormula} also admits the expression
\[\Phi(v_{1}\odot \cdots \odot v_{n})=\frac{1}{n!}\sum_{\sigma\in S_{n}} \sign \cdot (\tau(v_{\sigma(1)})\circ \tau(v_{\sigma(2)}) \circ \cdots \circ \tau(v_{\sigma(n)}))(1)\]
for all $n\geq 1$, 
where $v_{1},\ldots, v_{n}\in V$ are homogeneous, and $\sign=\sign(\sigma; v_{1},\ldots,v_{n})$ denotes the Koszul sign satisfying $v_{1}\odot \cdots\odot v_{n}=\sign \cdot v_{\sigma(1)}\odot \cdots \odot v_{\sigma(n)}$.
Indeed, the coderivation property shows that the expression above, together with $\Phi(1)=1$, defines a coalgebra morphism. Moreover, it is easy to check that the following relation holds:
\[\Phi(v_1\odot\cdots\odot v_{n+1})=\frac{1}{n+1}\sum_{i=1}^{n+1}
\sign_{i}\cdot \tau(v_i)\big(\Phi(v_1\odot\cdots\widehat{v_i}\cdots\odot v_{n+1})\big)\, ,\]
where $\varepsilon_{i}=(-1)^{\degree{v_{i}}(\degree{v_{1}}+\cdots+\degree{v_{i-1}})}$.
Projecting this identity onto $V$ recovers the recursive formula in Corollary~\ref{cor:MainFormula}.
\end{remark}

\subsection{Formal mappings}
Another classical linearisation problem is the linearisation of a map near a fixed point (see for instance~\cite{MR695786}).
We will show that an analogous theorem holds by replacing a formal vector field with an endomorphism of a formal manifold.

We begin by recalling the classical setting.
Let $g: M\to N$ be a morphism of smooth manifolds, and let $G:C^{\infty}(N)\to C^{\infty}(M)$ denote the induced morphism of algebras.
For each $p\in M$, there exists an induced map on the tangent spaces
\[T_{p}g: T_{p}M \to T_{g(p)}N \, ,\]
but it does not extend to a map between their spaces of vector fields, unless $g$ is a diffeomorphism. Instead, $g$ induces a pair of maps connecting the space of vector fields
\[ 
\XX(M)\xto{g_{\ast}} \Gamma(g^{\ast}TN) \xleftarrow{\iota_{g}} \XX(N) \, .
\]
Under the identification 
\begin{multline*}
\Gamma(g^{\ast}TN)\cong C^{\infty}(M)\tensor_{C^{\infty}(N)} \XX(N)  \\
\cong  \{ \chi \in \Hom(C^{\infty}(N), C^{\infty}(M)): \chi(uv)=\chi(u)G(v) + G(u)\chi(v), \, \forall u,v \in C^{\infty}(N)\} \, , 
\end{multline*}
the above maps are explicitly described by
\[ g_{\ast}(X): u\mapsto X(G(u)), \qquad \iota_{g}(Y): u\mapsto G(Y(u)) \]
for all $X\in \XX(M)$, $Y\in \XX(N)$ and $u\in C^{\infty}(N)$.
If $g$ is a diffeomorphism, then the map $\iota_{g}$ is invertible. In this case, we have an isomorphism of vector fields $g_{\star}: \XX(M) \to \XX(N)$
defined by 
\begin{equation}\label{eq:Fstar}
	 g_{\star}(X): u\mapsto G^{-1}(X (G(u)))
\end{equation}
for $u\in C^{\infty}(N)$.
The same is true for formal manifolds.

Let $V$ be a formal manifold, and let $g:V\to V$ be a formal endomorphism of $V$, corresponding to an algebra homomorphism $G:\widehat{S}(V^{\vee})\to \widehat{S}(V^{\vee})$. The linear part of $g$ is a linear map $g_{1}:V\to V$. We naturally view $g_{1}$ as a formal endomorphism of $V$, whose associated algebra homomorphism $G_{1}:\widehat{S}(V^{\vee})\to \widehat{S}(V^{\vee})$ is the unique continuous extension of the dual map $g^{\transpose}_{1}:V^{\vee}\to V^{\vee}$.
Note that, by the definition of morphisms of formal manifolds, $g$ necessarily preserves the origin. 
We say the endomorphism $g$ is \textbf{linearisable} if there exists a diffeomorphism $\psi:V\to V$ such that $T\psi|_{0}=\id_{T_{0}V}$ and 
\begin{equation}\label{eq:MLIN}
g_{1}=\psi^{-1} g \psi \, .
\end{equation}
Equation~\eqref{eq:MLIN} induces a commutative diagram:
\begin{equation} \label{diag:M1}
\begin{tikzcd}
	\XX(V) \arrow{d}{\psi_{\star}} \arrow{r}{(g_{1})_{\ast}} & \Gamma(g_{1}^{\ast}TV) \arrow{d}{\psi_{\star}} & \arrow{l}[swap]{\iota_{g_{1}}} \XX(V) \arrow{d}{\psi_{\star}}\\
	\XX(V) \arrow{r}{g_{\ast}} & \Gamma(g^{\ast}TV) & \arrow{l}[swap]{\iota_{g}} \XX(V)
\end{tikzcd}
\end{equation}
where $\psi_{\star}:\Gamma(g_{1}^{\ast}TV) \to \Gamma(g^{\ast}TV)$ is defined by the same formula as in Eq.~\eqref{eq:Fstar}.

Since $g_{1}$ is linear, for a constant vector field $v\in V\cong \XX_{\con}(V)$, it satisfies
$(g_{1})_{\ast}(v)= \iota_{g_{1}}(g_{1}(v))$.
Indeed, it suffices to check for a linear function $\xi \in V^{\vee}\cong \Hom(V,\KK)$ on $V$:
\[ ((g_{1})_{\ast}(v))(\xi) = v(G_{1}(\xi)) = \xi(g_{1}(v))=G_{1}(\xi(g_{1}(v)))=(\iota_{g_{1}}(g_{1}(v)))(\xi)\, . \]
Thus, the commutative diagram~\eqref{diag:M1} induces the following commutative diagram
\begin{equation} \label{diag:M2}
\begin{tikzcd}
V \arrow{d}{g_{1}} \arrow{r}{\tau} & \XX(V) \arrow{rd}{g_{\ast}} \\
V \arrow{r}{\tau} & \XX(V) \arrow{r}{\iota_{g}} & \Gamma(g^{\ast}TV) \\
\end{tikzcd}
\end{equation}
by setting $\tau=\psi_{\star}|_{V}$. As before, $\sigma_{0}:\XX(V)\to V$ is defined as in Eq.~\eqref{eq:Sig0}.

\begin{definition}
Let $V$ be a formal manifold and $g:V\to V$ be a formal endomorphism. We say $g$ satisfies the \textbf{splitting property} if there exists a linear map $\tau:V\to \XX(V)$ such that $\sigma_{0}\circ \tau=\id_{V}$ holds and the diagram~\eqref{diag:M2} is commutative.
\end{definition}

The following theorem shows that linearisability is equivalent to the splitting property.

\begin{theorem}
Let $V$ be a formal manifold, and let $g:V\to V$ be a formal endomorphism of $V$. Then $g$ is linearisable if and only if $g$ satisfies the splitting property.
\end{theorem}
As in the case of formal vector fields, we reformulate this result using the language of graded coalgebras.

\begin{theorem}\label{thm:Second}
Let $V$ be a graded vector space, and let $F:SV\to SV$ be a morphism of graded coalgebras with linear part $f_1$. The following are equivalent:
\begin{enumerate}[label=\rm{(\roman*)}]
\item \label{item:MA}
There exists an isomorphism of graded coalgebras
\[\Phi:SV\to SV\]
such that $\phi_{1}=\id_{V}$ and $\Phi f_{1}=F \Phi$.
\item \label{item:MB}
There exists a degree-preserving linear map $\tau:V\to \coDer(SV)$ such that $\tau(v)\cdot 1 = v$ and 
\[
 \tau(f_{1}(v))\cdot F(\bfv) = F \big( \tau(v) \cdot \bfv \big)
\]
 for all $v\in V$ and $\bfv\in SV$.
\end{enumerate}
\end{theorem}

\begin{proof}
The proof is mostly identical to the proof of Theorem~\ref{thm:First-1}.

Assume that~\ref{item:MA} holds. Define a linear map $\tau:V\to \coDer(SV)$ by
\[ \tau(v) \cdot \bfv = \Phi(v\odot \Phi^{-1}(\bfv)) \]
for $v\in V$ and $\bfv\in SV$. Then we have
\[ \tau(f_{1}(v))\cdot F(\bfv) = \Phi(f_{1}(v)\odot \Phi^{-1}(F(\bfv))) = \Phi f_{1}(v\odot \Phi^{-1}(\bfv)) =F\Phi(v\odot \Phi^{-1}(\bfv)) = F \big( \tau(v) \cdot \bfv \big) \, ,\]
as required.

Conversely, assume that~\ref{item:MB} holds. Denoting $F=(f_{1},f_{2},\ldots)$, we have
\begin{equation}\label{eq:MCB}
\sum_{k=1}^{n}\tau_{k}^{f_{1}(v)}f_{n}^{k}(\bfw)=\sum_{k=1}^{n+1}f_{k}\widetilde{\tau}_{n-k+1}^{v}(\bfw)
\end{equation}
for $v\in V$ and $\bfw\in S^{n}(V)$. As usual, we denote the composition $f_{n}^{k}:S^{n}(V)\into SV\xto{F} SV \onto S^{k}(V)$.

It suffices to show that the isomorphism $\Phi=(\phi_{1},\phi_{2},\ldots):SV\to SV$ of graded coalgebras
defined by $\phi_{1}=\id_{V}$ and the relation~\eqref{eq:PN1} satisfies $\Phi f_{1} = F \Phi$, or equivalently, 
\[\phi_{n} f_{1}= \sum_{k=1}^{n} f_{k}\phi_{n}^{k} \]
for all $n \geq 1$, by viewing $f_{1}:SV\to SV$ as a coalgebra morphism determined by $f_{1}:V\to V$. 

We argue by induction on $n$. Since $\phi_{1}=\id_{V}$, it is clear that when $n=1$, we have $\phi_{1}f_{1}=f_{1}=f_{1}\phi_{1}$.

Assume that $\phi_{r}f_{1} = \sum_{k=1}^{r}f_{k} \phi_{r}^{k}$ holds for all $r\leq n$. 
This means that for $\bfw \in S^{\leq n}(V)$, we have \[\Phi f_{1}(\bfw) = F\Phi(\bfw)\] and in particular, we have
\begin{equation}\label{eq:MPF}
 \phi_{n}^{k}f_{1}=\sum_{m=k}^{n}f_{m}^{k}\phi_{n}^{m}
 \end{equation}
for $k\leq n$.
By~\eqref{eq:PN1}, we have
\begin{align*}
(n+1)(\phi_{n+1}f_{1} - f_{1} \phi_{n+1})(\bfv)
&= \sum_{i=1}^{n+1}\sum_{k=1}^{n} \sign_{i} \cdot \big( \tau_{k}^{f_{1}(v_{i})} \phi_{n}^{k}(f_{1}(\bfv^{\{i\}})) -  f_{1}(\tau_{k}^{v_{i}}\phi_{n}^{k}(\bfv^{\{i\}})) \big) \\
&=\sum_{i=1}^{n+1}\sum_{k=1}^{n} \sum_{m=k}^{n} \sign_{i} \cdot \tau_{k}^{f_{1}(v_{i})}f_{m}^{k}\phi_{n}^{m}(\bfv^{\{i\}}) - \sum_{i=1}^{n+1}\sum_{k=1}^{n} \sign_{i} \cdot f_{1}( \tau_{k}^{v_{i}}\phi_{n}^{k}(\bfv^{\{i\}}))\\
&=\sum_{i=1}^{n+1}\sum_{k=1}^{n} \sign_{i} \cdot \big( \sum_{m=1}^{k} \tau_{m}^{f_{1}(v_{i})}f_{k}^{m}\phi_{n}^{k}(\bfv^{\{i\}}) - f_{1}(\tau_{k}^{v_{i}}\phi_{n}^{k}(\bfv^{\{i\}}))\big)\\
&=\sum_{i=1}^{n+1}\sum_{k=1}^{n}  \big( \sum_{m=2}^{k+1} \sign_{i} \cdot f_{m}\widetilde{\tau}_{k-m+1}^{v_{i}}\phi_{n}^{k}(\bfv^{\{i\}}) \big)\\
&=\sum_{i=1}^{n+1}\sum_{l=1}^{n}\sum_{k=0}^{n-l} \sign_{i}\cdot f_{l+1} (\widetilde{\tau}_{k}^{v_{i}} \phi_{n}^{k+l}(\bfv^{\{i\}}))\\
 &=(n+1)\sum_{l=1}^{n} f_{l+1}\phi_{n+1}^{l+1}(\bfv)
\end{align*}
for all $\bfv=v_{1}\odot \cdots \odot v_{n+1}\in S^{n+1}(V)$ and $\sign_{i}=(-1)^{\degree{v_{i}}(\degree{v_{1}}+\cdots+\degree{v_{i-1}})}$.
Here, the second equality holds by~\eqref{eq:MPF},
the fourth equality holds by~\eqref{eq:MCB},
and the last equality holds by the same argument as in~\eqref{eq:rPLUS1}.
This completes the proof.
\end{proof}

\subsection{Application to admissible vector fields}

In this section, we present sufficient conditions to guarantee the existence of a splitting map in Theorem~\ref{thm:First}. Then, we show that any admissible formal vector field, introduced by Basto--Gonçalves~\cite{MR2678982}, satisfies these sufficient conditions.

Let $V$ be a formal manifold, and
let $X$ be a formal vector field on $V$ with $\sigma_{0}(X)=0$. 
Thus, we may write
$X=X_{1}+X_{2}+\cdots$.
For the space $\Hom(V, \XX(V))\cong V^{\vee}\tensor \XX(V)$, we introduce the following two operations: for all $F\in \Hom(V, \XX(V))$ and $v\in V$, we define
\begin{enumerate}
\item the Lie derivative $\tliederivative{X_{1}}:\Hom(V, \XX(V)) \to \Hom(V, \XX(V))$ along $X_{1}$
by
\[ \tliederivative{X_{1}}F: v\mapsto [X_{1},F(v)]- F([X_{1},v])\, .\]
\item the Lie algebra $\XX(V)$-action $\XX(V)\times \Hom(V,\XX(V)) \to \Hom(V, \XX(V))$
by
\[ (Y\act F)(v)=[Y, F(v)] \]
for any $Y\in \XX(V)$.
\end{enumerate}

We denote by $\iota_{V}:V\to \XX(V)$ the natural inclusion $V\cong \XX_{\con}(V) \subset \XX(V)$.

The following proposition provides a criterion for linearisability of $X$.
\begin{proposition}\label{prop:LinQ}
Let $V$ be a formal manifold, and let $X$ be a formal vector field on $V$ satisfying $\sigma_{0}(X)=0$.
Suppose there exist subspaces $\fUV \subset \Hom(V, \XX(V))$ and $\fG \subset \XX(V)$ such that
\begin{enumerate}[label=\rm(\roman*)]
\item $\iota_{V} \in \fUV$,
\item $X_{n}\in \fG$ for all $n\geq 2$,
\item $\fG \act \fUV \subset \tliederivative{X_{1}}(\fUV)$.
\end{enumerate}
Then $X$ is linearisable.
\end{proposition}
\begin{proof}
By Theorem~\ref{thm:First}, it suffices to show that $X$ satisfies the splitting property. 
Note that a linear map $\tau:V\to \XX(V)$ decomposes into $\tau=\tau_{0}+\tau_{1}+\cdots$, where $\tau_{n}: V\to S^{n}(V^{\vee})\tensor V$. Then the map $\tau$ is a splitting map for $X$ if and only if
\begin{gather}
\tau_{0}(v)=\iota_{V}(v)=v \nonumber \\
 (\tliederivative{X_{1}} \tau_{n})(v)=[X_{1}, \tau_{n}(v)]-\tau_{n}([X_{1},v])=-\sum_{k=2}^{n+1}[X_{k},\tau_{n+1-k}(v)]= -\sum_{k=2}^{n+1} (X_{k}\act \tau_{n+1-k})(v) \, \label{eq:Tau1}
 \end{gather}
for all $n\geq 1$ and $v\in V$.

Set $\widetilde{\tau}_{0}=\iota_{V} \in \fUV$, and suppose inductively that there exist $\widetilde{\tau}_{k}\in \fUV$, $k < n$, satisfying the corresponding identities obtained from Eq.~\eqref{eq:Tau1} by replacing each $\tau_{j}$ with $\widetilde{\tau}_{j}$. 
Then we have
\[(X_{k}\act \widetilde{\tau}_{n+1-k}) \in \fG \act \fUV \subset \tliederivative{X_{1}}(\fUV)\]
for all $2\leq k \leq n+1$. Therefore, there exists $\widetilde{\tau}_{n}\in \fUV$ satisfying Eq.~\eqref{eq:Tau1} for $\widetilde{\tau}_{n}$. 
Continuing inductively, we obtain $\widetilde{\tau}_{n}$ for all $n\geq 0$ such that the same equation holds for every $n\geq 0$. 
Let $p_{n}:\Hom(V, \XX(V)) \to \Hom(V, S^{n}(V^{\vee})\tensor V)$ be the natural projection, and let $\tau_{n}:=p_{n}(\widetilde{\tau}_{n})$. 
Since $\tliederivative{X_{1}}$ stabilises $\Hom(V, S^{n}(V^{\vee})\tensor V)$, and similarly, $X_{k}\act \Hom(V, S^{n}(V^{\vee})\tensor V) \subset \Hom(V, S^{n+k-1}(V^{\vee})\tensor V)$,
the family $\{\tau_{n}\}_{n\geq 0}$ satisfies Eq.~\eqref{eq:Tau1}, hence yields a splitting map $\tau=\sum_{n \geq 0}\tau_{n}:V\to \XX(V)$ for $X$. This completes the proof.
\end{proof}

In~\cite{MR2678982}, Basto-Gonçalves introduced admissibility for formal vector fields and showed that such vector fields are always linearisable. Below, we briefly describe admissible vector fields and give an alternative proof using Proposition~\ref{prop:LinQ}. For this, we will assume that the base field $\KK=\CC$ is the field of complex numbers.

Let $V$ be a formal manifold with $\dim V= d$, and let $X$ be a formal vector field on $V$ with $\sigma_{0}(X)=0$.

Let $\lambda_{1},\lambda_{2},\ldots,\lambda_{m}$ denote the eigenvalues of the Lie derivative $\liederivative{X_{1}}=[X_{1},-]$ acting on the space of constant vector fields $\XX_{\con}(V)\cong V$.
Choose a basis $\{\partial_{1},\ldots,\partial_{d}\}$ of $V$ so that $\liederivative{X_{1}}$ is written in Jordan normal form:
\[
\liederivative{X_{1}} =
\begin{bmatrix}
	J_{d_1}(\lambda_1) & 0 & \cdots & 0 \\
	0 & J_{d_2}(\lambda_2) & \cdots & 0 \\
	\vdots & \vdots & \ddots & \vdots \\
	0 & 0 & \cdots & J_{d_m}(\lambda_m)
\end{bmatrix},
\]
where $J_{d_i}(\lambda_i)$ denotes the Jordan block of size $d_i$ associated with $\lambda_i$. 
We decompose $V$ as
\[ V= V_{1}\oplus \cdots \oplus V_{m} \, , \]
where $V_{i}$ denotes the invariant subspace corresponding to the $i$-th Jordan block $J_{i}$; that is, 
$\liederivative{X_{1}}(V_{i})\subset V_{i}$ and the restriction $\liederivative{X_{1}}|_{V_{i}}$ is represented by $J_{d_{i}}(\lambda_{i})$. 
Clearly, $\dim V_{i}=d_{i}$.
Let $\{x_{1},\ldots,x_{d}\}$ denote the basis of $V^{\vee}$ dual to the basis $\{\partial_{1},\ldots,\partial_{d}\}$ of $V$.

Set
\[
\Lambda=(\underbrace{\lambda_1,\ldots,\lambda_1}_{d_1},\ldots,\underbrace{\lambda_m,\ldots,\lambda_m}_{d_m})\in \CC^d.
\]
For a multi-index
$K=(k_1,\ldots, k_d)\in \ZZ_{\ge 0}^d$,
we define
\[
\degree{K}=\sum_{i=1}^{d} k_{i} \, , \qquad 
x^{K}= x_{1}^{k_{1}} x_{2}^{k_{2}} \cdots x_{d}^{k_{d}}.
\]
Then, for each $n\ge 1$, the $n$-th coefficient $X_n$ of $X$ can be written as
\[
X_n=\sum_{\degree{K}=n}\sum_{i=1}^{d} A_K^{i} \cdot x^K \partial_{i} \, ,
\qquad A_K^{i}\in \CC.
\]

We define the set of multi-indices corresponding to the non-zero nonlinear terms of $X$ as:
\[\cA = \{ K-e_{i} \in \ZZ^{d} \mid A_{K}^{i}\neq 0, \degree{K}\geq 2\} \, , \]
where $e_{i}$ denotes the standard basis vector with $1$ in the $i$-th position and $0$ elsewhere.
Next, we project these indices into $\CC$ using the dot product with eigenvalues:
\[ \widehat{\cA}= \{ \Lambda \cdot \alpha \in \CC \mid \alpha \in \cA\} \, .\]
Let $\widehat{\cC}$ be the set of all finite sums of elements in $\widehat{\cA}$ with positive integer coefficients.

\begin{definition}\cite[Definition 3.1]{MR2678982}
A formal vector field $X$ on $V$ satisfying $\sigma_{0}(X)=0$ is called \textbf{admissible} if $0\notin \widehat{\cC}$.
\end{definition}

For the rest of this section, we reserve the symbol $P$ for a tuple $(p_{1},\ldots,p_{m}) \in \ZZ^{m}_{\geq 0}$. By abuse of notation, we write $\Lambda \cdot P = \lambda_{1}p_{1}+\lambda_{2}p_{2}+\cdots+\lambda_{m}p_{m}$.

For each $p\ge 0$, the space $S^{p}(V^\vee)\otimes V$ admits the decomposition
\[
S^p(V^\vee)\otimes V
=\bigoplus_{\degree{P}=p}\ \bigoplus_{r=1}^m
V^{P}_{r},
\]
where
\[
V^{P}_{r}
=
S^{p_1}(V_{1}^\vee)\otimes \cdots \otimes S^{p_m}(V_{m}^\vee)\otimes V_{r}.
\]
Similarly,
\[
V^\vee\otimes S^p(V^\vee)\otimes V
\cong \Hom\big(V,S^p(V^\vee)\otimes V\big)
=
\bigoplus_{\degree{P}=p}\ \bigoplus_{q,r=1}^m
V^{(q; P)}_{r},
\]
where
\[
V^{(q; P)}_{r}
=
V_{q}^\vee \otimes S^{p_1}(V_{1}^\vee)\otimes \cdots \otimes S^{p_m}(V_{m}^\vee)\otimes V_{r}.
\]

Define
\[
\XX(V)_{\cC}=\prod_{p=2}^{\infty} \XX(V)^{p}_{\cC}, \qquad 
\XX(V)^p_{\cC}
=
\bigoplus_{\substack{\degree{P}=p\\ P\cdot \Lambda-\lambda_{r}\in \widehat{\cC}}}
V^{P}_{r},
\]
and
\[
(V^\vee\otimes \XX(V))_{\cC}= \prod_{p=1}^{\infty} (V^\vee\otimes \XX(V))^{p}_{\cC}, 
\qquad
(V^\vee\otimes \XX(V))^p_{\cC}
=
\bigoplus_{\substack{\degree{P}=p\\ P\cdot \Lambda +\lambda_{q}-\lambda_{r}\in\widehat{\cC}}}
V^{(q;P)}_{r}.
\]

\begin{lemma} \label{lem:Admissible}
Let $X$ be an admissible formal vector field on $V$. The following hold:
\begin{enumerate}
\item The $\KK$-linear map $\tliederivative{X_{1}}:(V^{\vee}\tensor \XX(V))\to (V^{\vee}\tensor \XX(V))$ restricts to an isomorphism
 \[\tliederivative{X_{1}}: (V^{\vee}\tensor \XX(V))_{\cC} \to (V^{\vee}\tensor \XX(V))_{\cC} \, . \]
\item If $Y\in \XX(V)_{\cC}$ and $F\in (V^{\vee}\tensor \XX(V))_{\cC}$, then
\[Y\act F \in (V^{\vee}\tensor \XX(V))_{\cC} \, .\]
\item The map
\[\XX(V)\to (V^{\vee}\tensor \XX(V)), \qquad  Y\mapsto (Y\act \iota_{V})\]
restricts to a map $\XX(V)_{\cC} \to (V^{\vee}\tensor \XX(V))_{\cC}$.
\end{enumerate}
\end{lemma}

\begin{proof}
For the first assertion, direct computation shows that each summand $V^{(q ; P)}_{r}$ is invariant under $\tliederivative{X_1}$, and that, for an appropriate choice of ordered basis of $V^{(q ; P)}_{r}$, the induced map $\tliederivative{X_1}:V^{(q ; P)}_{r}\to V^{(q ; P)}_{r}$ can be represented by an upper triangular matrix with constant diagonal entry
\[
 \lambda_r-\lambda_q -\Big(\sum_{i=1}^m p_i\lambda_i \Big)  \, .
\]
Thus, $\tliederivative{X_1}:V^{(q ; P)}_{r}\to V^{(q ; P)}_{r}$ is an isomorphism if and only if $\lambda_r-\lambda_q-\sum_{i=1}^m p_i\lambda_i\neq 0$. Since $X$ is admissible, $\tliederivative{X_{1}}$ restricts to an isomorphism
$\tliederivative{X_{1}}: (V^{\vee}\tensor \XX(V))_{\cC} \to (V^{\vee}\tensor \XX(V))_{\cC}$.

For the second assertion, let $F\in V^{(q;P)}_{r} $ and $Y\in V^{P'}_{s}$. 
It is straightforward to check that
\[ Y\act F \in V^{(q; P+ P' - e_{r})}_{s} + V^{(q; P - e_{s}+P')}_{r} \, .\]
Since $\widehat{\cC}$ is closed under addition, if the values $(P\cdot \Lambda +\lambda_{q}-\lambda_{r})$ and $(P' \cdot \Lambda  - \lambda_{s})$ belong to $\widehat{\cC}$, so is their sum $( (P+P')\cdot \Lambda +\lambda_{q}-\lambda_{r}-\lambda_{s})$.
This proves the second assertion. 

For the last assertion, observe that 
\[\iota_{V}=\sum_{j=1}^{d}x_{j}\tensor \partial_{j} \in \bigoplus_{k=1}^{m}V^{(k; 0,\ldots,0)}_{k}.\]
 Thus, if $Y\in V^{P}_{r}$, then 
\[Y\act \iota_{V}\in \bigoplus_{k=1}^{m}  V^{(k; P-e_{k})}_{r}  \,. \]
Therefore, if $Y \in \XX(V)_{\cC}$, then $Y\act \iota_{V} \in (V^{\vee}\tensor \XX(V))_{\cC}$. This proves the last assertion, hence the lemma.
\end{proof}

By the lemma above, it is straightforward to check that 
if $X$ is admissible, then $\fUV = \langle \iota_{V} \rangle \oplus (V^\vee\otimes \XX(V))_{\cC}$ and $\fG= \XX(V)_{\cC}$ satisfy the hypotheses in Proposition~\ref{prop:LinQ}. Here, $\langle \iota_{V} \rangle$ denotes the vector space spanned by $\{\iota_{V}\}$.
Consequently, Proposition~\ref{prop:LinQ} provides an alternative proof of the following theorem due to Basto-Gonçalves~\cite{MR2678982}, demonstrating the utility of the splitting criterion established in Theorem~\ref{thm:First}.
\begin{theorem}\cite[Theorem~1]{MR2678982}\label{thm: admissible vector field is linearizable}
	Every admissible formal vector field $X$ is linearisable.
\end{theorem}

\begin{remark}
The notion of a weakly admissible vector field, which generalises admissibility, was introduced in the same paper~\cite{MR2678982}, where it was shown that such vector fields are always linearisable. By Theorem~\ref{thm:First}, any weakly admissible vector field must admit a splitting map. However, it is not immediately clear whether weakly admissible vector fields generally satisfy the conditions in Proposition~\ref{prop:LinQ}.
\end{remark}

\section{Linearisations for homotopy algebras} \label{Section:Variation on the Linearisation problem}

\subsection{L-infinity algebras}

Applying Theorem~\ref{thm:First-1} to an $L_{\infty}[1]$ algebra $(V,Q)$, we obtain an elementary proof of the following theorem, originally due to Bandiera~\cite{MR3622306}, providing a geometric interpretation.
\begin{theorem}[\cite{MR3622306}]\label{thm:Bandiera}
Let $(V,Q)$ be an $L_{\infty}[1]$ algebra. The following are equivalent.
\begin{enumerate}[label=\rm{(\roman*)}]
\item \label{item:LLIN}
There exists an $L_{\infty}[1]$ algebra isomorphism $\Phi=(\phi_{1},\phi_{2},\cdots):(V,q_{1})\to (V,Q)$ with $\phi_{1}=\id_{V}$.
\item \label{item:LSPLIT}
There exists a DG right inverse of the chain map $\ev_{1}:(\coDer(SV), [Q,-]) \to (V,q_{1})$.
\end{enumerate}
\end{theorem}

\begin{remark}
To provide context for the $A_{\infty}$ analogue of the above theorem, we note that Bandiera~\cite{MR3622306} proved that conditions~\ref{item:LLIN} and~\ref{item:LSPLIT} are equivalent to a third condition: the natural (Chevalley--Eilenberg) spectral sequence computing $H_{\CE}(V,V)$ degenerates at $E_{1}$.
\end{remark}

\begin{remark}
An $L_{\infty}[1]$ algebra $(V,Q)$ is called \textit{abelian} if $Q$ is linear (i.e., $Q=q_{1}$), and $(V,Q)$ is called \textit{homotopy abelian} if it is $L_{\infty}[1]$ quasi-isomorphic to an abelian $L_{\infty}[1]$ algebra. It is well known that an $L_{\infty}[1]$ algebra is homotopy abelian if and only if it is linearisable (see, for instance~\cite[Lemma~12.7.2]{MR4485797}).
\end{remark}

\begin{remark}
Although we are working over a base field $\KK$, the proof of Theorem~\ref{thm:First-1} also applies to the setting of $L_{\infty}[1]$ algebras over a DG commutative algebra, established in~\cite{arXiv:2509.26341}.
\end{remark}

\subsection{A-infinity algebras}
\subsubsection{linearisation of A-infinity algebras}
Motivated by the proof of Theorem~\ref{thm:First-1}, we establish an $A_{\infty}[1]$ analogue of Theorem~\ref{thm:Bandiera}.

Let $(A,\mu)$ be an $A_{\infty}[1]$ algebra. Then the $A_{\infty}[1]$ algebra structure $\mu\in \coDer(TA)$ is uniquely determined by the sequence $\{\mu_{k}\}_{k\geq 1}$ of maps $\mu_{k}:T^{k}(A)\to A$ of degree $1$ satisfying a compatibility condition.
Consider the space $\Hom(TA,A)$ equipped with a map $D_{\mu}:\Hom(TA,A)\to \Hom(TA,A)$ of degree $1$ defined by
\[D_{\mu}(F):\mathbf{a} \mapsto \mu_{n+1}(F(1)\tensor \mathbf{a}) + \sum_{k=1}^{n}\mu_{n-k+1}(F(a_{1}\tensor\cdots \tensor a_{k})\tensor a_{k+1}\tensor \cdots \tensor a_{n}) -(-1)^{\degree{F}} F\mu(\mathbf{a})\]
for $F\in \Hom(TA,A)$ and $\mathbf{a}=a_{1}\tensor\cdots\tensor a_{n}\in T^{n}(A)$.
Equivalently, using the natural convolution product $\conv$ on $\Hom(TA,TA)$, the map $D_{\mu}$ can be formulated as
\[D_{\mu}(F)=\pr_{A} \mu(F\conv \id) - (-1)^{\degree{F}}F\mu \, ,\]
where $\pr_{A}:TA \to A$ denotes the natural projection. 
It can be checked that the $A_\infty[1]$ relations for $\{\mu_k\}_{k \geqslant 1}$ imply that $D_\mu^2=0$, i.e., it is a differential.
Moreover, the evaluation at $1\in \KK =T^{0}(A) \subset TA$ yields a cochain map
\[ \ev_{1}:(\Hom(TA,A),D_{\mu}) \to (A, \mu_{1}) \, .\]

\begin{remark}\label{rem:AAR}
We point out that the cohomology group of the cochain complex $(\Hom(TA,A),D_{\mu})$ is the Hochschild cohomology $\HH(A,A^{R})$ of the $A_{\infty}[1]$ algebra $A$ with coefficients in the $(A,A)$-bimodule $A^{R}$. Here, $A^{R}=A$ as a graded vector space, endowed with an  $A_{\infty}[1]$ bimodule structure
\[ \rho_{n,m}: A^{\tensor n}\tensor A^{R} \tensor A^{\tensor m} \to A^{R} \, ,\]
such that $\rho_{0,m}=\mu_{m+1}$ and $\rho_{n,m}=0$ for all $n\geq 1$.
Thus the left action of $A$ on $A^{R}$ is trivial, while the right action of $A$ on $A^{R}$ is the standard right action. This is different from the usual Hochschild cohomology $\HH(A,A)$ of the $A_{\infty}[1]$ algebra $A$, where the coefficient $A$ is the canonical $(A,A)$-bimodule defined by
\[ \mu_{n,m}: A^{\tensor n}\tensor A \tensor A^{\tensor m} \to A\]
satisfying $\mu_{n,m}=\mu_{n+m+1}$ for all $n,m$. Indeed, the usual Hochschild cohomology $\HH(A,A)$ is the cohomology group of the cochain complex $(\coDer(TA), \liederivative{\mu})$ where the Lie derivative $\liederivative{\mu}$ is defined by the graded commutator with $\mu$.
\end{remark}

Note that the space $\Hom(TA,A)$ carries a natural decreasing filtration by subcomplexes
\[\mathcal{F}^0\supset \mathcal{F}^1\supset \cdots \]
in which
\begin{equation}\label{eq:AFilt}
\mathcal{F}^{p}=(\Hom(T^{\geq p}(A), A),D_\mu)=(\Hom(\bigoplus_{k=p}^{\infty}A^{\tensor k},A),D_\mu).
\end{equation}
Equipped with this filtration, $(\Hom(TA,A), D_\mu)$ becomes a filtered cochain complex, and thus induces a (Hochschild) spectral sequence computing $\HH(A, A^{R})$~\cite{MR1269324}.

It is straightforward to check the following lemma.
\begin{lemma}\label{lem:AABDiff}
Let $(A,\mu)$ be an $A_{\infty}[1]$ algebra. 
The convolution product with the identity map $\id:TA\to TA$ defines a morphism of cochain complexes
\[ \iota_{A}: (\Hom(TA,A), D_{\mu}) \to (\Hom(TA,TA), \liederivative{\mu}), \qquad \iota_{A}(F)=F\conv \id \, ,\]
where $\liederivative{\mu}$ is defined as the graded commutator with $\mu$. That is, for $F\in \Hom(TA,A)$,
\[ D_{\mu}(F) \conv \id = \mu (F\conv \id) -(-1)^{\degree{F}} (F\conv \id) \mu : TA \to TA \, .\]
\end{lemma}

Recall that an isomorphism of $A_{\infty}[1]$ algebras $\Phi:(A,\mu) \to (B,\nu)$ is an isomorphism of the graded tensor coalgebras
\[ \Phi:(TA,\Delta_{A})\to (TB,\Delta_{B}) \]
such that $\nu\Phi=\Phi\mu$.

\begin{lemma}\label{lem:AABIso}
Let $\Phi:(A,\mu) \to (B,\nu)$ be an isomorphism of $A_{\infty}[1]$ algebras. 
Then, for $F\in \Hom(TA,A)$, the assignment 
\[ F\mapsto \pr_{B}\Phi (\iota_{A}(F)) \Phi^{-1} \]
defines an isomorphism of cochain complexes
\[\widetilde{\gamma}_{\Phi}: (\Hom(TA,A), D_{\mu}) \to (\Hom(TB,B), D_{\nu}) \, .\]
Moreover, $\widetilde{\gamma}_{\Phi}$ is an isomorphism of filtered cochain complexes with respect to the filtration~\eqref{eq:AFilt}.
\end{lemma}
\begin{proof}
Denote by $\gamma_{\Phi}:(\Hom(TA,TA),\liederivative{\mu}) \to (\Hom(TB,TB),\liederivative{\nu})$ the morphism of cochain complexes defined as the conjugation by $\Phi$, i.e., $\gamma_{\Phi}(\widetilde{F})= \Phi \widetilde{F} \Phi^{-1}$ for $\widetilde{F}\in \Hom(TA,TA)$. 

Direct computation (see the remark below for a structural perspective) shows that 
the following diagram is commutative:
\[ 
\begin{tikzcd}
(\Hom(TA,TA), \liederivative{\mu}) \arrow{r}{\gamma_{\Phi}} & (\Hom(TB,TB), \liederivative{\nu})\\
(\Hom(TA,A), D_{\mu})\arrow{u}{\iota_{A}} \arrow{r}{\widetilde{\gamma}_{\Phi}} & (\Hom(TB,B),D_{\nu}) \arrow{u}{\iota_{B}}\\
\end{tikzcd}
\]
Since $\pr_{B}\iota_{B}$ is the identity on $\Hom(TB,B)$, and by Lemma~\ref{lem:AABDiff}, we have
\[ D_{\nu}\widetilde{\gamma}_{\Phi}=\pr_{B}\iota_{B}D_{\nu}\widetilde{\gamma}_{\Phi}=\pr_{B}\liederivative{\nu}\gamma_{\Phi}\iota_{A}=\pr_{B}\gamma_{\Phi}\iota_{A}D_{\mu}=\widetilde{\gamma}_{\Phi} D_{\mu}\, . \]
The statement for the filtration follows immediately from the fact that $\Phi(T^{\leq n}(A))\subset T^{\leq n}(B)$.
This completes the proof.
\end{proof}

\begin{remark}
The commutativity of the diagram in the proof can also be seen structurally using the notion of comodules, rather than direct computation; it relies on the following four observations:
\begin{enumerate}
\item The map $\iota_{A}$ factors through $\Hom(TA, A\tensor TA)\subset \Hom(TA,TA)$, by viewing $A\tensor TA$ as $T^{\geq1}(A)$.
\item Since $TA$ is a coalgebra, both $\KK\tensor TA =TA$ and $A\tensor TA$ are right cofree comodules. 
\item Since $\Phi:TA\to TB$ is an isomorphism of coalgebras, if $G:TA\to A\tensor TA$ is a morphism of right $TA$-comodules, then $\gamma_{\Phi}(G):TB\to B\tensor TB$ is a morphism of right $TB$-comodules.
\item Since $A\tensor TA$ is cofree, any morphism of right $TA$-comodules $G:TA\to A\tensor TA$ satisfies $G=\iota_{A}(\pr_{A}G)$.
\end{enumerate}
\end{remark}

The following theorem is an $A_{\infty}[1]$ analogue of Theorem~\ref{thm:Bandiera}. Note that the symbol $A^{R}$ denotes the $A_{\infty}[1]$ $(A,A)$-bimodule described in Remark~\ref{rem:AAR}.
\begin{theorem}\label{thm:ATFAE}
Let $(A,\mu)$ be an $A_{\infty}[1]$ algebra. With the notation above, the following are equivalent.
\begin{enumerate}[label=\rm{(\roman*)}]
\item \label{item:AA} 
There is an $A_{\infty}[1]$ algebra isomorphism $\Phi=(\phi_{1},\phi_{2},\ldots): (A, \mu_{1}) \to (A, \mu)$ with $\phi_{1}=\id$.
\item \label{item:AB} 
There exists a DG right inverse $\tau$ of the cochain map $\ev_{1}:(\Hom(TA,A),D_{\mu}) \to (A, \mu_{1})$. 
\item \label{item:AC}
The natural (Hochschild) spectral sequence computing $\HH(A,A^{R})$ degenerates at $E_{1}$.
\end{enumerate}
\end{theorem}
We say that an $A_{\infty}[1]$ algebra $(A,\mu)$ satisfies the \textbf{splitting property} if the condition~\ref{item:AB} of Theorem~\ref{thm:ATFAE} holds. Thus, $A_{\infty}[1]$ algebra is linearisable if and only if it satisfies the splitting property.

\begin{proof}
We first show that \ref{item:AA} $\Leftrightarrow$ \ref{item:AB}. 

Assume that \ref{item:AA} holds. Define a map
$\tau: A \to \Hom(TA,A)$
by
\[\tau(a) \mathbf{a} = \pr_{A} \Phi(a\tensor \Phi^{-1}(\mathbf{a})) \, ,\]
where $a\in A$ and $\mathbf{a}\in TA$. It is straightforward to check that $\ev_{1} (\tau(a))=a$ and 
\begin{eqnarray*}
\tau(\mu_{1}(a))\mathbf{a} &=&\pr_{A} \Phi(\mu_{1}(a)\tensor \Phi^{-1}(\mathbf{a})) \\
&=& \pr_{A} \Phi\mu_{1}(a\tensor \Phi^{-1}(\mathbf{a})) -(-1)^{\degree{a}} \pr_{A} \Phi (a\tensor \mu_{1} \Phi^{-1}(\mathbf{a})) \\
&=&\pr_{A} \mu \Phi (a\tensor \Phi^{-1}(\mathbf{a})) -(-1)^{\degree{a}} \pr_{A} \Phi (a\tensor \Phi^{-1}\mu( \mathbf{a}))\\
&=& \pr_{A}\mu (\tau(a)\conv \id)\mathbf{a}  - (-1)^{\degree{a}} \tau(a)(\mu(\mathbf{a})) \\
&=& D_{\mu}(\tau(a))\mathbf{a} \, .
\end{eqnarray*}
Here, the fourth equality holds since, if $\Phi^{-1}(\mathbf{a})=\mathbf{b}=b_{1}\tensor \cdots \tensor b_{n} \in T^{n}(A)$, 
\[ \Phi(a\tensor \mathbf{b}) = \phi_{1}(a)\tensor \Phi(\mathbf{b}) +\phi_{n+1}(a\tensor \mathbf{b}) + \sum_{k=2}^{n} \phi_{k}(a\tensor b_{1}\tensor \cdots \tensor b_{k-1})\tensor \Phi(b_{k}\tensor \cdots \tensor b_{n}) \, \]
and $\Delta(\mathbf{b})=(\Phi^{-1}\tensor \Phi^{-1})\Delta \mathbf{a}$; the case of arbitrary $\Phi^{-1}(\mathbf{a})$ follows by linearity.
This shows that \ref{item:AA} $\Rightarrow$ \ref{item:AB}. 

Conversely, assume \ref{item:AB} holds. For each $a\in A$, we write $\tau(a)=\sum_{n=0}^{\infty}\tau_{n}^{a}$ where $\tau_{n}^{a}:T^{n}(A)\to A$.
Then the relation $\tau(\mu_{1}(a))=D_{\mu}(\tau(a))$ reads
\begin{equation}\label{eq:ATAU}
\tau_{n}^{\mu_{1}(a)}
= \sum_{k=0}^{n}\mu_{n-k+1}(\tau_{k}^{a}\conv (\id_{A}^{\tensor n-k})) 
- (-1)^{\degree{a}} \sum_{k=1}^{n} \tau_{k}^{a} \widetilde{\mu}_{n-k+1} \, .
\end{equation}

Similarly to the proof of Theorem~\ref{thm:First-1}, we claim that the $A_{\infty}[1]$ isomorphism 
\[\Phi=(\phi_{1},\phi_{2},\ldots) : (TA, \mu_{1}) \to (TA, \mu)\]
is defined inductively by $\phi_{1}=\id_{A}$ and 
\begin{equation}\label{eq:AFN+1}
\phi_{n+1}(a \tensor \mathbf{a})=\sum_{k=1}^{n}\tau_{k}^{a} \phi_{n}^{k}(\mathbf{a})
\end{equation}
where $\mathbf{a}\in T^{n}(A)$ and $\phi_{n}^{k}:T^{n}(A)\into TA \xto{\Phi} TA \to T^{k}(A)$. 
We use induction to prove the claim. 
Since $\phi_{1}=\id_{A}$, we have $\phi_{1}\widetilde{\mu}_{1}=\mu_{1}\phi_{1}$. 
Assume that for $m\leq n$, 
\[\phi_{m}\widetilde{\mu}_{1}= \sum_{k=1}^{m}\mu_{k}\phi_{m}^{k} \]
holds. Then, for any $\mathbf{b}\in T^{m}(A)$, we have $\Phi \widetilde{\mu}_{1}(\mathbf{b})=\mu \Phi (\mathbf{b})$, and in particular, for any $k \leq m$, we have
\begin{equation}\label{eq:AFMK}
 \phi_{m}^{k}\widetilde{\mu}_{1} = \sum_{r=k}^{m} \widetilde{\mu}_{r-k+1}\phi_{m}^{r} \, . 
\end{equation}
Now, similarly to the proof of Theorem~\ref{thm:First-1}, by~\eqref{eq:ATAU},~\eqref{eq:AFN+1}, and~\eqref{eq:AFMK}, direct computation shows that
\begin{eqnarray*}
 (\phi_{n+1}\widetilde{\mu}_{1}-\mu_{1}\phi_{n+1})(a\tensor \mathbf{a}) 
&=& \sum_{k=1}^{n} 
(\tau_{k}^{\mu_{1}(a)} \phi_{n}^{k}+(-1)^{\degree{a}} \tau_{k}^{a}\phi_{n}^{k}\widetilde{\mu}_{1} - \mu_{1}\tau_{k}^{a} \phi_{n}^{k})(\mathbf{a})\\
&=& \sum_{k=1}^{n} (\tau_{k}^{\mu_{1}(a)}+(-1)^{\degree{a}}\big(\sum_{r=1}^{k}\tau_{r}^{a}\widetilde{\mu}_{k-r+1}\big) - \mu_{1}\tau_{k}^{a})\phi_{n}^{k}(\mathbf{a}) \\
&=&\sum_{k=1}^{n}\sum_{r=0}^{k-1}\mu_{k-r+1}(\tau_{r}^{a}\conv (\id_{A}^{\tensor k-r}))\phi_{n}^{k}(\mathbf{a})\\
&=&\sum_{k=1}^{n}\big( \mu_{k+1}(a\tensor \phi_{n}^{k}(\mathbf{a})) + \sum_{r=1}^{k-1}\sum_{s=r}^{n-k+r} \mu_{k-r+1}(\tau_{r}^{a}\phi_{s}^{r}\conv \phi_{n-s}^{k-r})(\mathbf{a}) \big)\\
&=& \sum_{k=1}^{n}\mu_{k+1}\phi^{k+1}_{n+1}(a\tensor \mathbf{a})
\end{eqnarray*}
for $a\in A$ and $\mathbf{a}\in T^{n}(A)$. This proves \ref{item:AB}~$\Rightarrow$~\ref{item:AA}.

By Lemma~\ref{lem:AABIso}, an isomorphism of $A_{\infty}[1]$ algebras induces an isomorphism of filtered cochain complexes. Therefore, it induces an isomorphism of the spectral sequence at every page, proving \ref{item:AA} $\Rightarrow$ \ref{item:AC}.

Finally, we show \ref{item:AC}~$\Rightarrow$~\ref{item:AB}. 
Let $\overline{T}A = \bigoplus_{n=1}^{\infty}A^{\otimes n}$.
Then we have a short exact sequence of cochain complexes
\[ 0\to (\Hom(\overline{T}A,A), D_{\mu}) \to (\Hom(TA,A), D_{\mu}) \to (A, \mu_{1}) \to 0 \, ,\]
and it induces a long exact sequence
\[\to H^{\bullet}(\Hom(\overline{T}A,A), D_{\mu}) \to H^{\bullet}(\Hom(TA,A), D_{\mu}) \to H^{\bullet}(A, \mu_{1}) \xto{[\delta]} H^{\bullet+1}(\Hom(\overline{T}A,A), D_{\mu}) \to\]
where the connecting homomorphism $[\delta]$ is induced by the cochain map
\[ \delta : (A,\mu_{1})\to (\Hom(\overline{T}A,A)[1], -D_{\mu}) \]
defined by
\[ \delta (a)= D_{\mu}(a)-\mu_{1}(a) \, . \]

It is well-known (cf.~\cite[Lemma~A.42]{MR2393625}) that the condition~\ref{item:AC} implies the vanishing of the connecting homomorphism $[\delta]$.
Therefore, there exists a degree-preserving map $s:A\to \Hom(\overline{T}A,A)$ such that $\delta(a)=D_{\mu}(s(a)) - s(\mu_{1}(a))$. 
Viewing the element $a\in A$ as an element of $\Hom(\KK,A)\subset \Hom(TA,A)$, we define $\tau:A\to \Hom(TA,A)$ by $\tau(a)=a - s(a)$. Then $\tau$ is a DG right inverse of $\ev_{1}$, thus establishing~\ref{item:AB}.
This completes the proof.
\end{proof}

\begin{example}
Let $A$ be a graded associative algebra equipped with a degree $1$ operator $d:A\to A$ satisfying $d^{2}=0$, which is not necessarily a derivation. Börjeson~\cite{MR3624036} constructs an $A_{\infty}[1]$ structure $\mu=(\mu_{1},\mu_{2},\ldots)$ on the graded vector space $A$ as follows: $\mu_{1}=d$, $\mu_{2}(a,b)=d(ab)-d(a)b-(-1)^{\degree{a}}a\, d(b)$, and, for $n\geq 2$,
\[ \mu_{n+1}(a,b_{1},\ldots,b_{n}) = \mu_{n}(ab_{1},b_{2},\ldots,b_{n}) - (-1)^{\degree{a}}a\,\mu_{n}(b_{1},\ldots,b_{n})\, , \]
where $a,b,b_{1},\ldots,b_{n}\in A$ are homogeneous and the symbol $ab$ denotes the original product of $a$ and $b$ in $A$.

We show that $(A,\mu)$ satisfies the splitting property. Define a map $\tau:A\to\Hom(TA,A)$ of degree $0$ by
\[ \tau(a)(1)=a \, ,\qquad \tau(a)(b)=-ab \, ,\qquad  \tau(a)(b_{1},\ldots,b_{n})=0 \quad \forall n\geq 2\, . \]

Clearly, $\ev_{1}\circ\tau=\id_{A}$. Moreover, it is straightforward to check that $D_{\mu}(\tau(a)) = \tau(d(a))$. Therefore, $\tau$ is a DG right inverse of $\ev_{1}$.
By Theorem~\ref{thm:ATFAE}, $(A,\mu)$ is linearisable, recovering the result established by Markl~\cite{MR3385702} through an explicit coalgebra automorphism.
\end{example}

As a special case, the linearisability of a DG algebra is characterised as follows.
\begin{corollary}
A DG algebra $(A,d,m)$ is linearisable if and only if
there exist maps $\phi_{n}:A\tensor A^{\tensor n}\to A$ of degree $(-n)$ for $n\geq 1$ such that
\[ 
m(a, b)=d\phi_{1}(a\tensor b)+\phi_{1}(da\tensor b) +(-1)^{\degree{a}}\phi_{1}(a\tensor db)
\]
and for $n\geq 2$, 
\begin{align*}
& m(\phi_{n-1}(a\tensor b_{1}\tensor \cdots \tensor b_{n-1}), b_{n}) - (-1)^{n-1} \sum_{i=1}^{n-1}(-1)^{i}\phi_{n-1}(a\tensor b_{1}\tensor \cdots \tensor m(b_{i},b_{i+1}) \tensor \cdots \tensor b_{n}) \\
&\begin{multlined}[0.92\displaywidth]
= d(\phi_{n}(a\tensor b_{1}\tensor \cdots \tensor b_{n})) - (-1)^{n}\phi_{n}(da\tensor b_{1}\tensor \cdots \tensor b_{n}) \\
-(-1)^{n} \sum_{i=1}^{n}(-1)^{\degree{a}+\degree{b_{1}}+\cdots +\degree{b_{i-1}}} \phi_{n}(a\tensor b_{1}\tensor \cdots \tensor db_{i}\tensor \cdots \tensor b_{n})
\end{multlined}
\end{align*}
for homogeneous $a, b, b_{1},\ldots, b_{n} \in A$.
\end{corollary}

\begin{remark}
As in the $L_{\infty}[1]$ algebra case, it is well known that an $A_{\infty}[1]$ algebra $(A,\mu)$ is linearisable if and only if $(A,\mu)$ is $A_{\infty}[1]$ quasi-isomorphic to a linear $A_{\infty}[1]$ algebra $(B,\nu)$ (i.e., $(B,\nu)$ is merely a cochain complex viewed as an $A_{\infty}[1]$ algebra). Equivalently, $(A,\mu)$ is linearisable if and only if its minimal model is trivial (i.e., for all $n\geq 2$, the $n$-ary products on $H(A)$ transferred from the $A_{\infty}[1]$ structure on $A$ vanish).

If $(A,\mu)$ is unital and $H(A)\neq0$, then $H(A)$ is a nonzero
unital graded algebra. In particular, its induced binary
multiplication is nonzero, so its minimal model cannot be trivial.
Consequently, a unital $A_\infty[1]$ algebra with nonzero cohomology
cannot be linearisable.
\end{remark}
\begin{remark}
Analogously to Remark~\ref{rem:AltForm}, the $A_{\infty}[1]$ algebra isomorphism $\Phi$ arising from the splitting map $\tau$ (see Eq.~\eqref{eq:AFN+1} for the recursive formula) admits the expression
\[ \Phi(a_{1}\tensor \cdots \tensor a_{n})=(T_{a_{1}}\circ \cdots \circ T_{a_{n}})(1)\, ,\]
for all $n\geq 1$ and $a_{1},\ldots,a_{n}\in A$, where $T_{a}(\mathbf{b})=(\tau(a)\conv \id_{TA})(\mathbf{b})$ for all $a\in A$ and $\mathbf{b}\in TA$.
\end{remark}

\subsubsection{Comparison with coderivations}
As noted in Remark~\ref{rem:AAR}, the cochain complex $(\Hom(TA,A),D_{\mu})$ appearing in Theorem~\ref{thm:ATFAE} computes the one-sided Hochschild cohomology $\HH(A,A^R)$, whereas the usual Hochschild cohomology $\HH(A,A)$ is computed by $(\coDer(TA), \liederivative{\mu})$.
In this section, we characterise when the evaluation map
\[ \ev_{1}: (\coDer(TA),\liederivative{\mu})\to (A,\mu_{1}) \]
admits a DG right inverse. 
We thus consider an analogue of the splitting property, obtained by replacing $(\Hom(TA,A),D_{\mu})$ with $(\coDer(TA),\liederivative{\mu})$.

Let $(A,\mu)$ be an $A_{\infty}[1]$ algebra. It is well known that there exists a decomposition of $(A,\mu)$ as an $A_{\infty}[1]$ algebra: $(A,\mu)\cong (K,d_{K})\oplus (H(A),\nu)$, where $(K,d_{K})$ is a contractible cochain complex (seen as a contractible linear $A_{\infty}[1]$ algebra), $H(A)$ is the cohomology of the cochain complex $(A,\mu_{1})$, and $\nu=(0,\nu_{2},\nu_{3},\ldots)$ is a minimal $A_{\infty}[1]$ structure on $H(A)$ (i.e., $\nu_{1}=0$). The minimal $A_{\infty}[1]$ algebra $(H(A),\nu)$ is called a \textbf{minimal model} of $(A,\mu)$.

Let $(B,\nu)$ be a minimal $A_{\infty}[1]$ algebra, and let $\HH(B)=H(\coDer(TB),\liederivative{\nu})$ denote the Hochschild cohomology of $(B,\nu)$. Since the evaluation map $\ev_{1}:(\coDer(TB),\liederivative{\nu}) \to (B,0)$ is a cochain map, it induces a map on cohomology 
\[ H(\ev_{1}):\HH(B)\to B \, .\]
According to Briggs--Gélinas~\cite{arXiv:1702.00721}, the $A_{\infty}[1]$ centre $Z_{\infty}(B)$ of a minimal $A_{\infty}[1]$ algebra $(B,\nu)$ is defined by the image of $H(\ev_{1})$:
\[ Z_{\infty}(B):= \operatorname{im}(H(\ev_{1})) \subset B \, .\]
If $Z_{\infty}(B)=B$, then the minimal $A_{\infty}[1]$ algebra $(B,\nu)$ is called $A_{\infty}[1]$ commutative.

We characterise the coderivation analogue of the `splitting property' for an $A_{\infty}[1]$ algebra.
\begin{proposition}\label{prop:CTA}
Let $(A,\mu)$ be an $A_{\infty}[1]$ algebra, and let $(H(A),\nu)$ be a minimal model. The following are equivalent.
\begin{enumerate}
\item The minimal model $(H(A),\nu)$ is $A_{\infty}[1]$ commutative in the sense of Briggs--Gélinas:
\[ H(A)=Z_{\infty}(H(A))\, .\]
\item There exists a DG right inverse $\tau$ of the cochain map $\ev_{1}:(\coDer(TA),\liederivative{\mu})\to (A,\mu_{1})$.
\end{enumerate}
\end{proposition}

\begin{proof}
Observe that $H(A)=Z_{\infty}(H(A))$ if and only if $H(\ev_{1}):\HH(H(A))\to H(A)$ is surjective. 
The invariance of Hochschild cohomology under $A_{\infty}[1]$ quasi-isomorphisms is compatible with the maps induced by evaluation~\cite{arXiv:1702.00721}. Consequently, $H(\ev_{1}):\HH(H(A))\to H(A)$ is surjective if and only if $H(\ev_{1}):\HH(A,A)\to H(A)$ is surjective.

On the cochain level, we have a short exact sequence:
\begin{equation}\label{eq:SESTA}
0\to (\coDer(\overline{T}A),\liederivative{\mu}) \to (\coDer(TA),\liederivative{\mu})\xto{\ev_{1}} (A,\mu_{1}) \to 0 \, .
\end{equation}
As in the proof of \ref{item:AC}~$\Rightarrow$~\ref{item:AB} in Theorem~\ref{thm:ATFAE}, we have a cochain map
$\delta:(A,\mu_{1})\to (\coDer(\overline{T}A)[1], -\liederivative{\mu})$
defined by $\delta(a)=\liederivative{\mu}(a)-\mu_{1}(a)$. Moreover, if the connecting homomorphism $[\delta]=0$, then there exists a DG right inverse $\tau$ of $\ev_{1}$. 

Now, assume $H(\ev_{1})$ is surjective. Then $[\delta]=0$, and hence a DG right inverse of $\ev_{1}$ exists. Conversely, if a DG right inverse of $\ev_{1}$ exists, then the short exact sequence~\eqref{eq:SESTA} splits, thus $H(\ev_{1})$ is surjective. This completes the proof.
\end{proof}

Given an $A_{\infty}[1]$ algebra $(A,\mu)$, the graded symmetrisation of $\mu$ produces an $L_{\infty}[1]$ algebra, namely $(A_{\Lie},q)$. By Briggs--Gélinas~\cite{arXiv:1702.00721}, if $(B,\nu)$ is an $A_{\infty}[1]$ commutative minimal $A_{\infty}[1]$ algebra, then the induced $L_{\infty}[1]$ algebra structure on $B_{\Lie}$ is abelian. Since graded symmetrisation preserves quasi-isomorphisms, we obtain the following corollary.
\begin{corollary}
Assume that $(A,\mu)$ satisfies either of the equivalent conditions in Proposition~\ref{prop:CTA}.
Then the induced $L_{\infty}[1]$ algebra $(A_{\Lie}, q)$ obtained by graded symmetrisation of $(A,\mu)$ is homotopy abelian.
\end{corollary}

\begin{example}
Consider an $A_{\infty}[1]$ algebra $(A,\mu)$ arising from an ordinary (not necessarily commutative) algebra. Then $A$ is concentrated in degree $-1$ and $\mu=\mu_{2}:A\tensor A\to A$. Note that $A$ is commutative if and only if $\mu_{2}(a\tensor b)-\mu_{2}(b\tensor a)=0$ for all $a,b\in A$. 
Assume that there exists a morphism of cochain complexes
\[\check{\tau}:(A,0) \to (\coDer(TA),\liederivative{\mu})\]
satisfying $\check{\tau}(a):1\mapsto a$. Then $A$ is commutative because
\[\liederivative{\mu}(\check{\tau}(a))(b) = \mu_{2}(a\tensor b)-\mu_{2}(b\tensor a)=0 \, . \]
Conversely, assume that $A$ is commutative. Then the map $\check{\tau}$ satisfying $\check{\tau}(a)(1)=a$ and
\begin{align*}
\check{\tau}(a) : (b_{1}\tensor \cdots \tensor b_{n}) \mapsto a\tensor b_{1}\tensor \cdots \tensor b_{n} &+  (-1)^{n} b_{1}\tensor \cdots \tensor b_{n}\tensor a\\
&+\sum_{i=1}^{n-1} (-1)^{i} b_{1}\tensor \cdots \tensor b_{i}\tensor a \tensor b_{i+1} \tensor \cdots \tensor b_{n}
\end{align*}
for all $n\geq 1$ is a morphism of cochain complexes. This recovers Proposition~\ref{prop:CTA} for ordinary associative algebras.
However, any $A_{\infty}[1]$ algebra $(A,\mu=\mu_{2})$ arising from an ordinary associative algebra is linearisable if and only if $\mu_{2}=0$. 
\end{example}

\begin{remark}\label{rem:Last}
The analogue of the spectral sequence condition~\ref{item:AC} in Theorem~\ref{thm:ATFAE} for $\HH(A,A)$ is not equivalent to the conditions in Proposition~\ref{prop:CTA}.
In fact, the equivalence of~\ref{item:AA} and~\ref{item:AC} in Theorem~\ref{thm:ATFAE} remains valid if one replaces $\HH(A,A^{R})$ in Theorem~\ref{thm:ATFAE}~\ref{item:AC} with $\HH(A,A)$; one can use the formality criteria in~\cite{MR3936682} for the operad of associative algebras, together with the $A_{\infty}$ version of~\cite[Lemma~6.1]{MR3353026}.
\end{remark}

\section*{Acknowledgements} 
We would like to thank Ping Xu and Ruggero Bandiera for helpful comments and discussions. The authors are grateful to the University of Genova (Seol and Wang), Pennsylvania State University (Seol and Wang), Sapienza Università di Roma (Seol), and Korea Institute for Advanced Study (Wang) for their hospitality during part of this work.

\section*{Declaration of generative AI and AI-assisted technologies in the manuscript preparation process}

During the preparation of this work, the authors used ChatGPT to improve the language and exposition and to assist in reviewing the mathematical arguments for correctness. The authors reviewed and edited the material as needed and take full responsibility for the content of the published article.

\appendix

\section{Graded coalgebras and graded coderivations}\label{sec:DGcoAlg}

In this appendix, we review the basics of graded coalgebras, $L_{\infty}$ algebras, and $A_{\infty}$ algebras, establishing several technical lemmas required in the main text. We refer the reader to~\cite{MR1786197, MR2954392, MR4485797} for a comprehensive introduction to coalgebras.

\subsection{Graded coalgebras}

A \textbf{graded coalgebra} $(C,\Delta,\epsilon)$ is a graded vector space $C$ equipped with degree-preserving linear maps $\Delta: C\rightarrow C\otimes C$, called the comultiplication, and $\epsilon:C\rightarrow \KK$, called the counit, satisfying
\begin{enumerate}[label=\rm{(\roman*)}]
\item $(\Delta\otimes \id_{C})\circ \Delta = (\id_{C}\otimes \Delta)\circ \Delta$,
\item $\mu_{\KK,C} \circ (\epsilon\otimes \id_{C})\circ \Delta = \id_{C} = \mu_{C,\KK} \circ (\id_{C}\otimes \epsilon)\circ \Delta,$
\end{enumerate}
where the ground field $\KK$ is considered as a graded vector space concentrated in degree zero, and $\mu_{\KK,C}, \mu_{C,\KK}$ are the scalar multiplications. 
A graded coalgebra $(C,\Delta,\epsilon)$ is said to be \textbf{graded cocommutative} if $\tw \circ \Delta = \Delta$, where $\tw: C \otimes C \to C\otimes C$ is defined by $\tw(x \otimes y) = (-1)^{\degree{x}\degree{y}} y \otimes x$.

For simplicity of notation, the graded coalgebra $(C,\Delta,\epsilon)$ is often denoted by $(C,\Delta)$.

Let $(B,\Delta_{B},\epsilon_{B})$ and $(C, \Delta_{C}, \epsilon_{C})$ be graded coalgebras. A morphism of graded coalgebras is a degree-preserving linear map $\Phi:B\to C$ satisfying
\begin{enumerate}[label=\rm{(\roman*)}]
\item $\Delta_{C} \circ \Phi = (\Phi\tensor \Phi)\circ \Delta_{B} $
\item $\epsilon_{C} \circ \Phi = \epsilon_{B}$ .
\end{enumerate}

Given a graded coalgebra $(C,\Delta)$, a morphism of graded vector spaces $p:C\to V$ is called a \textbf{cogenerator} if
\[(\epsilon, p, p^{\tensor 2}\Delta, p^{\tensor 3}\Delta^{2},\cdots) : C \to \KK\times \prod_{n=1}^{\infty}V^{\tensor n}\]
is injective. 

\begin{lemma}\cite[Proposition~11.1.13]{MR4485797} \label{lem:MorCogen}
A morphism of graded coalgebras $\Phi:(B,\Delta_{B})\to (C,\Delta_{C})$ is uniquely determined by its composition $p \Phi : B\to V$ with a cogenerator $p: C \to V$.
\end{lemma}

A \textbf{coderivation} of degree $i$ on a graded coalgebra $(C,\Delta)$  is a linear map $Q: C \to C$ of degree $i$ such that 
\[
\Delta\circ Q = (\id_{C}\otimes \, Q+ Q\otimes \id_{C})\circ \Delta \, .
\]
The space of coderivations of degree $i$ on $C$ is denoted by $\coDer^i(C)$, and $\coDer(C):= \bigoplus_i \coDer^i(C)$. 
\begin{lemma} \cite[Lemma 11.2.8]{MR4485797} \label{lem:CoderCogen}
Let $(C,\Delta)$ be a graded coalgebra and let $p:C\to V$ be a cogenerator. A coderivation $Q: C\to C$ is uniquely determined by its composition $pQ:C\to V$ with the cogenerator $p$.
\end{lemma}

We introduce the notion of convolution product to simplify the notation below.

Let $(C, \Delta, \epsilon)$ be a graded coalgebra, and suppose we are given a unital associative algebra $(A,\mu, 1_{A})$.
The \textbf{convolution product} $\conv: \Hom(C,A) \times \Hom(C,A) \to \Hom(C,A)$ is defined by
\begin{equation}\label{eq:ConvProd}
f\conv g := \mu (f\otimes g)\Delta \, ,
\end{equation}
where $f,g \in \Hom(C,A)$. One can check that  $(\Hom(C,A),\conv )$ forms a graded algebra with the unit $1_{\Hom(C,A)} \in \Hom(C,A)$, defined by
\[
1_{\Hom(C,A)}(x):= \epsilon(x) \cdot 1_A
\] 
for $x\in C$. 

We note that, by the definition of coderivations, we have
\[(f\conv g)Q = (-1)^{\degree{g} \degree{Q}} fQ\conv g + f\conv gQ\]
for homogeneous elements $f,g\in \Hom(C,A)$ and a homogeneous coderivation $Q\in \coDer(C)$.

\begin{remark}
Given a coalgebra $C$, the $\KK$-linear dual $C^{\vee}=\Hom(C,\KK)$ is naturally an associative algebra. However, unless an algebra $A$ is finite dimensional, the $\KK$-linear dual of $A$ does \textit{not} necessarily inherit the coalgebra structure via the multiplication. Similarly, if $Q$ is a homogeneous coderivation on $C$, then the induced map $Q^{\transpose}:C^{\vee}\to C^{\vee}$ defined by $\xi \mapsto -(-1)^{\degree{\xi}\degree{Q}}\xi Q$ for $\xi\in C^{\vee}$ is a derivation, but the converse is not true in general. 
\end{remark}

\subsection{Tensor coalgebras}

Let $V$ be a graded vector space. The \textbf{tensor coalgebra} of $V$ is the graded vector space \[TV := \bigoplus_{n=0}^{\infty} V^{\tensor n}\] where $V^{\otimes 0}:=\KK$. It is equipped with the counit defined by the projection $\epsilon:TV \onto \KK$, and the comultiplication
$\Delta$ defined by $\Delta(1)=1\tensor 1$, $\Delta(v)= 1\tensor v + v \tensor 1$, and
\[
\Delta(v_{1}\tensor \cdots \tensor v_{n}) 
= 1\tensor (v_{1}\tensor \cdots \tensor v_{n}) + (v_{1}\tensor \cdots \tensor v_{n})\tensor 1 
+ \sum_{i=1}^{n-1} (v_{1}\tensor \cdots \tensor v_{i})\tensor (v_{i+1}\tensor \cdots \tensor v_{n})
\]
for any $v, v_{1}, \ldots, v_{n} \in V$.

This tensor coalgebra $(TV,\Delta)$ is the  cofree conilpotent coalgebra cogenerated by $V$ with  the natural projection $p: TV\to V$ serving as its cogenerator. 

Furthermore, the space $TV$ carries a natural associative algebra structure $\mu:TV\times TV \to TV$ defined by the concatenation:
\[\mu(\mathbf{v}, \mathbf{w})=\mathbf{v}\tensor \mathbf{w}\]
for $\mathbf{v}, \mathbf{w}\in TV$. Consequently, the space of linear endomorphisms $\Hom(TV,TV)$ admits a convolution product $\conv$.

\subsubsection{Endomorphisms} \label{sec:EndoT}
According to Lemma~\ref{lem:MorCogen},
an endomorphism $\Phi:(TV,\Delta) \to (TV,\Delta)$ of the tensor coalgebra $(TV,\Delta)$ is uniquely determined by the maps
\[\phi_{n}:=p\Phi|_{T^{n}(V)} :T^{n}(V)\to V\]
for $n\geq 1$, where $T^n V=V^{\otimes n}$. Indeed, given $\phi_{n}:T^{n}(V) \to V$ for $n\geq 1$, the endomorphism $\Phi:TV\to TV$ is determined by $\Phi(1)=1$ and the components
\[\phi_{n}^{m}:T^{n}(V)\into TV \xto{\Phi} TV \onto T^{m}(V)\]
defined by
\[\phi_{n}^{m}=\sum_{r_{1}+\cdots + r_{m}=n} (\phi_{r_{1}}\conv \cdots \conv \phi_{r_{m}}) =\sum_{r=s}^{n-m+s} (\phi_{r}^{s}\conv \phi_{n-r}^{m-s})\, \]
for each $1\leq s\leq m$.
More explicitly, for $v_{1},\ldots, v_{n}\in V$,
\begin{eqnarray*}
\phi_{n}^{m}(v_{1}\tensor \cdots \tensor v_{n}) &=& \sum_{r_{1}+\cdots + r_{m}=n} \phi_{r_{1}}(v_{1}\tensor \cdots \tensor v_{r_{1}}) \tensor \cdots \tensor \phi_{r_{m}}(v_{n-r_{m}+1}\tensor \cdots \tensor v_{n}) \\
&=& \sum_{r=s}^{n-m+s}\phi_{r}^{s}(v_{1}\tensor \cdots \tensor v_{r})\tensor \phi_{n-r}^{m-s}(v_{r+1}\tensor \cdots \tensor v_{n}) \, .
\end{eqnarray*}

\subsubsection{Coderivations}
According to Lemma~\ref{lem:CoderCogen}, 
a coderivation $Q\in \coDer(TV,TV)$ of degree $k$ on $TV$ is uniquely determined by
\[q_{n}=pQ|_{T^{n}(V)}: T^{n}V\to V\]
for $n\geq 0$. Indeed, given $q_{n}:T^{n}(V)\to V$  with  $\degree{q_{n}}=k$ for $n\geq 0$, let $q=\sum_{n=0}^{\infty} q_{n} : TV\to V $. 
Then we obtain a coderivation $Q$ of degree $k$  as follows: for $v_{1},\ldots, v_{n}\in V$, 
\begin{align*} Q(v_{1}\tensor \cdots \tensor v_{n}) =&\sum_{i=0}^{n} (-1)^{k(\sum_{r=1}^i\degree{v_{r}})} \cdot  v_1\otimes \cdots \otimes v_i \otimes q_0(1)\otimes \cdots \otimes v_n \\
	&+ \sum_{j=1}^{n}  \sum_{i=0}^{n-j}(-1)^{k(\sum_{r=1}^i\degree{v_{r}})} \cdot  v_1\cdots \otimes v_{i}\otimes q_{j}(v_{i+1}\tensor \cdots \tensor v_{i+j})\tensor v_{i+j+1}\tensor \cdots \tensor v_{n}. 
	\end{align*}

\subsection{Symmetric coalgebras}\label{sec:coAlgSV}

Let $V$ be a graded vector space. The \textbf{symmetric coalgebra} of $V$ is the graded vector space $SV := \bigoplus_{n=0}^\infty V^{\odot n}$ equipped with the counit defined by the projection $\epsilon : SV \onto S^{0}(V) \cong \KK$, and 
the comultiplication
$\Delta$ defined by $\Delta(1) = 1 \otimes 1$, $\Delta(v) = 1 \otimes v + v \otimes 1$, and
\begin{multline*}
\Delta(v_{1} \odot \cdots \odot v_{n}) =1\tensor (v_{1} \odot \cdots \odot v_{n}) + (v_{1} \odot \cdots \odot v_{n})\tensor 1 \\
+ \sum_{k=1}^{n-1} \sum_{\sigma \in \shuffle(k,n-k)} \sign \cdot (v_{\sigma(1)} \odot \cdots \odot v_{\sigma(k)}) \otimes (v_{\sigma(k+1)} \odot \cdots \odot v_{\sigma(n)}) \, ,
\end{multline*}
for any $v, v_{1}, \ldots, v_{n} \in V$. Here, $\sign = \sign(\sigma; v_{1},\ldots, v_{n})=\pm 1$ is determined by the Koszul sign convention. 

The symmetric coalgebra $(SV,\Delta)$ is the cofree conilpotent cocommutative coalgebra cogenerated by $V$. 
Indeed, the natural projection $p:SV\to V$ serves as its cogenerator.

Analogous to the case of $TV$, the space $SV$ carries a natural associative algebra structure $\mu:SV\times SV\to SV$ defined by the symmetric tensor product $\odot$:
\[\mu(\mathbf{v}, \mathbf{w})= \mathbf{v}\odot \mathbf{w}\]
and thus, $\Hom(SV,SV)$ admits a convolution product $\conv$. Note that by the graded cocommutativity of $\Delta$ and the graded commutativity of $\mu$, the convolution product $\conv$ is graded commutative.

\subsubsection{Endomorphisms}
According to Lemma~\ref{lem:MorCogen},
an endomorphism $\Phi:(SV,\Delta) \to (SV,\Delta)$ of the symmetric coalgebra $(SV,\Delta)$ is uniquely determined by 
\[\phi_{n}:=p\Phi|_{S^{n}(V)} :S^{n}(V)\to V\]
for $n\geq 1$. Explicitly, given $\phi_{n}:S^{n}(V) \to V$ for $n\geq 1$, the endomorphism $\Phi:SV\to SV$ is determined by $\Phi(1)=1$ and the components
\[\phi_{n}^{m}:S^{n}(V)\into SV \xto{\Phi} SV \onto S^{m}(V)\]
defined by
\[\phi_{n}^{m}=\frac{1}{m!}\sum_{r_{1}+\cdots + r_{m}=n} (\phi_{r_{1}}\conv \cdots \conv \phi_{r_{m}}) = \frac{1}{m} \sum_{r=1}^{n-m+1} (\phi_{r}\conv \phi_{n-r}^{m-1}) \, .\]

We further investigate the relations between the components $\phi_{n}^{m}$ of an endomorphism $\phi$ of $(SV,\Delta)$.

Consider a set
\[R_{n}^{m}:=\{ (r_{1},\ldots, r_{m}) \in \ZZ_{>0}^{m} : r_{1}+\cdots +r_{m}=n, \quad r_{1}\leq \cdots \leq r_{m}\}. \]
Associated to each $\mathbf{r}=(r_{1},\ldots, r_{m})\in R_{n}^{m}$, there exist a positive integer $l\in \ZZ_{>0}$ and positive integers $i_{1},\ldots,i_{l} \in \ZZ_{>0}$ such that $i_{1}+\cdots + i_{l}=m$ and
\[r_{1}=\cdots = r_{i_{1}} < r_{i_{1}+1} = \cdots = r_{i_{1}+i_{2}} < \cdots < r_{m-i_{l}+1} = \cdots = r_{m} .\]
Denoting by $I_{\mathbf{r}}=(i_{1},\ldots,i_{l})$, we define 
\[\phi_{\mathbf{r}}=\phi_{r_{1}}\conv \cdots \conv \phi_{r_{m}} \, , \qquad c_{I_{\mathbf{r}}} = \frac{1}{i_{1}! \cdot i_{2}! \cdots i_{l}!} \, .\]

\begin{lemma}\label{lem:PNMSym}
With the notation above, we have
\[ \phi_{n}^{m}= \sum_{\mathbf{r}\in R_{n}^{m}} c_{I_{\mathbf{r}}} \cdot \phi_{\mathbf{r}} \, .\]
\end{lemma}
\begin{proof}
Since the convolution product $\conv$ on $\Hom(SV,SV)$ is graded commutative, it is straightforward to check that
\[\phi_{n}^{m}=\frac{1}{m!}\sum_{r_{1}+\cdots + r_{m}=n} (\phi_{r_{1}}\conv \cdots \conv \phi_{r_{m}}) = \sum_{\mathbf{r}\in R_{n}^{m}}c_{I_{\mathbf{r}}} \cdot \phi_{\mathbf{r}} \, .\]
This completes the proof.
\end{proof}

\begin{lemma}\label{lem:PNM}
The following equation holds:
\[ \phi_{n}^{m}=\frac{1}{n}\sum_{r=1}^{n-m+1} r \cdot (\phi_{r}\conv \phi_{n-r}^{m-1}) \, . \]
\end{lemma}
\begin{proof}
We may denote an element $\mathbf{r}\in R_{n}^{m}$ with $I_{\mathbf{r}}=(i_{1},\ldots,i_{l})$ by
\[\mathbf{r}=(\overbrace{\widetilde{r}_{1},\ldots, \widetilde{r}_{1}}^{\text{$i_{1}$ times}}, \overbrace{\widetilde{r}_{2},\ldots,\widetilde{r}_{2}}^{\text{$i_{2}$ times}}, \cdots , \overbrace{\widetilde{r}_{l},\cdots,\widetilde{r}_{l}}^{\text{$i_{l}$ times}}) .\]
where $\widetilde{r}_{1}<\cdots < \widetilde{r}_{l}$. Then $n=i_{1}\widetilde{r}_{1}+ \cdots + i_{l}\widetilde{r}_{l}$.

For each $j=1,\ldots, l$, denote by $\mathbf{r}^{\{j\}}\in R_{n-\widetilde{r}_{j}}^{m-1}$ an element obtained by omitting one $\widetilde{r}_{j}$ from $\mathbf{r}$. That is,
\[\mathbf{r}^{\{j\}}=(\overbrace{\widetilde{r}_{1},\ldots, \widetilde{r}_{1}}^{\text{$i_{1}$ times}}, \ldots, \overbrace{\widetilde{r}_{j},\ldots,\widetilde{r}_{j}}^{\text{$(i_{j}-1)$ times}}, \cdots , \overbrace{\widetilde{r}_{l},\cdots,\widetilde{r}_{l}}^{\text{$i_{l}$ times}}) .\]
For this fixed $\mathbf{r}\in R_{n}^{m}$, we have
\[n\cdot c_{I_{\mathbf{r}}}= \frac{\widetilde{r}_{1}}{(i_{1}-1)! \cdot i_{2}! \cdots i_{l}!} + \frac{\widetilde{r}_{2}}{i_{1}! \cdot (i_{2}-1)! \cdots i_{l}!}+ \cdots + \frac{\widetilde{r}_{l}}{i_{1}! \cdot i_{2}! \cdots (i_{l}-1)!} = \sum_{j=1}^{l} \widetilde{r}_{j}\cdot c_{I_{\mathbf{r}^{\{j\}}}}.\]

Under this notation, we have
\[\phi_{\mathbf{r}}=\phi_{\widetilde{r}_{1}}\conv \phi_{\mathbf{r}^{\{1\}}}=\phi_{\widetilde{r}_{2}}\conv \phi_{\mathbf{r}^{\{2\}}}=\cdots = \phi_{\widetilde{r}_{l}}\conv \phi_{\mathbf{r}^{\{l\}}} \]
and thus, we have
\[ n\cdot c_{I_{\mathbf{r}}} \cdot \phi_{\mathbf{r}} =  \sum_{j=1}^{l} \widetilde{r}_{j}\cdot c_{I_{\mathbf{r}^{\{j\}}}} \cdot (\phi_{\widetilde{r}_{j}}\conv \phi_{\mathbf{r}^{\{j\}}}) \, . \]

Together with Lemma~\ref{lem:PNMSym}, we have
\[\sum_{r=1}^{n-m+1} r \cdot (\phi_{r}\conv \phi_{n-r}^{m-1} )
=\sum_{r=1}^{n-m+1} \sum_{\widetilde{\mathbf{r}}\in R_{n-r}^{m-1}} r \cdot c_{I_{\widetilde{\mathbf{r}}}} \cdot (\phi_{r}\conv \phi_{\widetilde{\mathbf{r}}}) 
= \sum_{\mathbf{r}\in R_{n}^{m}} n \cdot c_{I_{\mathbf{r}}} \cdot \phi_{\mathbf{r}}
= n \cdot \phi_{n}^{m} .\]
This completes the proof.
\end{proof}

\subsubsection{Coderivations}
According to Lemma~\ref{lem:CoderCogen}, 
a coderivation $Q\in \coDer(SV,SV)$ on $SV$ is uniquely determined by
\[q_{n}=pQ|_{S^{n}(V)}: S^{n}(V)\to V\]
for $n\geq 0$. Indeed, given $q_{n}:S^{n}(V)\to V$ for $n\geq 0$, let $q=\sum_{n=0}^{\infty} q_{n} : SV\to V (\subset SV)$. 
Then we obtain a coderivation $Q=q \conv \id_{SV}:SV\to SV$. 
More explicitly, for $v_{1},\ldots, v_{n}\in V$, 
\[ Q(v_{1}\odot \cdots \odot v_{n}) = q_{0}(1)\odot v_{1}\odot \cdots \odot v_{n} + \sum_{k=1}^{n} \sum_{\sigma \in \shuffle(k,n-k)} \sign \cdot q_{k}(v_{\sigma(1)}\odot \cdots \odot v_{\sigma(k)})\odot v_{\sigma(k+1)}\odot \cdots \odot v_{\sigma(n)} \, ,\]
where $\shuffle(k,n-k)$ denotes the set of $(k,n-k)$-shuffles and
$\sign = \sign (\sigma; v_{1},\cdots,v_{n})=\pm 1$ is determined by the Koszul sign convention.

For a fixed $k$, let $q_{k}:S^{k}(V)\to V$. We denote by $\widetilde{q}_{k}$ the coderivation on $SV$ determined by $q_{k}:S^{k}(V)\to V$.
In the following lemmas, we consider its relation with the component $\phi_{n}^{m}:S^{n}(V)\to S^{m}(V)$ of an endomorphism $\Phi:SV\to SV$ of $(SV,\Delta)$. The component $\phi_{n}^{m}$ can be viewed as an element of $\Hom(SV,SV)$ by the zero extension.

\begin{lemma}\label{lem:PNKQ}
The following equality holds:
\[ \phi_{n}^{m} \widetilde{q}_{k} =\sum_{r=1}^{n-m+1} (\phi_{r} \widetilde{q}_{k}\conv \phi_{n-r}^{m-1}) : S^{n+k-1}(V)\to S^{m}(V)\]
for $m\geq 2$ and $n\geq m$.
\end{lemma}
\begin{proof}
We use induction on $m$. When $m=2$, we have
\begin{equation*}
\phi_{n}^{2} \widetilde{q}_{k}
=\frac{1}{2}\sum_{r=1}^{n-1} (\phi_{r} \conv \phi_{n-r}) \widetilde{q}_{k} 
=\frac{1}{2}\sum_{r=1}^{n-1} (\phi_{r} \widetilde{q}_{k} \conv \phi_{n-r}+\phi_{r}\conv \phi_{n-r} \widetilde{q}_{k}) 
= \sum_{r=1}^{n-1} (\phi_{r} \widetilde{q}_{k} \conv \phi_{n-r}) \, ,
\end{equation*}
where the last equality holds since $\conv$ is commutative.

Fix $m\geq 2$, and assume that
\[ \phi_{n}^{m-1} \widetilde{q}_{k} =\sum_{r=1}^{n-m+2} (\phi_{r} \widetilde{q}_{k}\conv \phi_{n-r}^{m-2})\]
holds for all $n\geq m-1$. Then, for any $n\geq m$, we have
\begin{multline*}
\frac{1}{m}\sum_{r=1}^{n-m+1} (\phi_{r} \conv \phi^{m-1}_{n-r} \widetilde{q}_{k})
= \frac{1}{m}\sum_{r=1}^{n-m+1} \sum_{s=1}^{n-r-m+2} (\phi_{r} \conv \phi_{s} \widetilde{q}_{k} \conv \phi_{n-r-s}^{m-2})\\
=\frac{1}{m}\sum_{s=1}^{n-m+1}\sum_{r=1}^{n-s-m+2} (\phi_{s}\widetilde{q}_{k}\conv \phi_{r}\conv \phi_{n-r-s}^{m-2})
=\frac{m-1}{m}\sum_{s=1}^{n-m+1} (\phi_{s} \widetilde{q}_{k}\conv \phi_{n-s}^{m-1}) \, ,
\end{multline*}
where the second equality holds since $\conv$ is commutative.
Therefore, we have
\[
\phi_{n}^{m} \widetilde{q}_{k}
=\frac{1}{m}\sum_{r=1}^{n-m+1}(\phi_{r}\conv \phi^{m-1}_{n-r}) \widetilde{q}_{k} \\
=\frac{1}{m}\sum_{r=1}^{n-m+1} (\phi_{r} \widetilde{q}_{k}\conv \phi^{m-1}_{n-r}+\phi_{r}\conv \phi^{m-1}_{n-r} \widetilde{q}_{k}) 
=\sum_{r=1}^{n-m+1} (\phi_{r} \widetilde{q}_{k}\conv \phi_{n-r}^{m-1}) \, .
\]
By the induction argument, this completes the proof.
\end{proof}

Similarly, we have the following relation.

\begin{lemma}\label{lem:QNPNK}
The following equation holds:
\[\widetilde{q}_{k} \phi_{n}^{m+k} =\sum_{r=k}^{n-m} (q_{k}\phi_{r}^{k} \conv \phi_{n-r}^{m}) : S^{n}(V)\to S^{m+1}(V)  \]
for all $n\geq m+k$.
\end{lemma}
\begin{proof}
For each pair $(r,s)$, define $\Delta_{r,s}: S^{r+s}(V)\to S^{r}(V)\tensor S^{s}(V)$ by the composition with the natural inclusion and projections
\[ \Delta_{r,s}:S^{r+s}(V)\into SV\xto{\Delta} SV \tensor SV \onto S^{r}(V)\tensor S^{s}(V) \, . \]
Then, it is straightforward to check that
\begin{eqnarray*} 
\widetilde{q}_{k}\phi_{n}^{m+k} &=& \mu( q_{k}\tensor \id_{S^{m}(V)})\Delta_{k,m} \phi_{n}^{m+k} \\
&=& \sum_{r=k}^{n-m} \mu (q_{k}\tensor \id_{S^{m}(V)}) \circ (\phi_{r}^{k}\tensor \phi_{n-r}^{m}) \Delta_{r,n-r}
\\
&=& \sum_{r=k}^{n-m} \mu (q_{k}\phi_{r}^{k}\tensor \phi_{n-r}^{m}) \Delta_{r,n-r}\\
&=&\sum_{r=k}^{n-m} (q_{k}\phi_{r}^{k}\conv \phi_{n-r}^{m}) \, .
\end{eqnarray*}
This completes the proof.
\end{proof}

\subsection{Homotopy algebras}

\subsubsection{$A_\infty$-algebras}

An \textbf{$A_\infty[1]$ algebra} $(A,\mu)$ consists of a graded vector space $A$ together with a degree $1$ coderivation $\mu$ on the tensor coalgebra $TA$ such that $\mu(1)=0$ and $\mu^2=0$. Equivalently, $(TA,\mu)$ is a differential graded coalgebra satisfying $\mu(1)=0$.

A graded vector space $V$ is said to carry an $A_\infty$-algebra structure if its desuspension $V[1]$ admits the structure of an $A_\infty[1]$ algebra.

By Lemma~\ref{lem:CoderCogen}, any degree $1$ coderivation $\mu$ on $TA$ with $\mu(1)=0$ is uniquely determined by a family of linear maps
\[
\mu_n \colon A^{\otimes n} \to A, \qquad n \ge 1,
\]
each of degree $\degree{\mu_{n}}=1$. In terms of these components, the condition $\mu^2=0$ is equivalent to the relations
\[
\sum_{i+j+k=n} \mu_{n-j+1} \circ (\id^{\otimes i} \otimes \mu_j \otimes \id^{\otimes k}) = 0,
\qquad n \ge 1.
\]
These identities provide an equivalent formulation of the notion of an $A_\infty[1]$ algebra.

Let $(A,\mu)$ and $(A',\mu')$ be $A_\infty[1]$ algebras. A morphism of $A_\infty[1]$ algebras from $(A,\mu)$ to $(A',\mu')$ is a morphism of differential graded coalgebras
\[
\Phi \colon (TA,\mu) \to (TA',\mu'),
\]
that is, a graded coalgebra morphism $\Phi \colon TA \to TA'$ satisfying $\Phi \circ \mu = \mu' \circ \Phi$.

\subsubsection{$L_\infty$-algebras}

An \textbf{$L_\infty[1]$ algebra} $(L,Q)$ consists of a graded vector space $L$ together with a degree $1$ coderivation $Q$ on the symmetric coalgebra $SL$ such that $Q(1)=0$ and $Q^2=0$. Equivalently, $(SL,Q)$ is a differential graded cocommutative coalgebra satisfying $Q(1)=0$.

A graded vector space $V$ is said to carry an $L_\infty$-algebra structure if its desuspension $V[1]$ admits the structure of an $L_\infty[1]$ algebra.

By Lemma~\ref{lem:CoderCogen}, any degree $1$ coderivation $Q$ on $SL$ with $Q(1)=0$ is uniquely determined by a family of linear maps
\[
q_n \colon S^nL \to L, \qquad n \ge 1,
\]
each of degree $\degree{q_{n}}=1$. In terms of these components, the condition $Q^2=0$ is equivalent to the relations
\[
\sum_{k=1}^n \sum_{\sigma\in \sh(k,n-k) }\sign \cdot q_{n-k+1}(q_k(v_{\sigma(1)} \odot \cdots \odot v_{\sigma(k)})\odot v_{\sigma(k+1)}\odot \cdots \odot v_{\sigma(n)}) = 0,
\qquad n \ge 1,
\]
for any $v_1,\cdots ,v_n\in L$. Here, the symbol $\sign=\sign(\sigma; v_{1},\ldots,v_{n})$ denotes the Koszul sign satisfying the relation $v_{1}\odot\cdots\odot v_{n}=\sign\cdot v_{\sigma(1)}\odot \cdots \odot v_{\sigma(n)}$.
These identities provide an equivalent formulation of the notion of an $L_\infty[1]$ algebra.

Let $(L,Q)$ and $(L',Q')$ be $L_\infty[1]$ algebras. A morphism of $L_\infty[1]$ algebras from $(L,Q)$ to $(L',Q')$ is a morphism of differential graded coalgebras
\[
\Phi \colon (SL,Q) \to (SL',Q'),
\]
that is, a graded coalgebra morphism $\Phi \colon SL \to SL'$ satisfying $\Phi \circ Q = Q'\circ \Phi$.

\printbibliography
\end{document}